\documentclass[12pt]{article}
\usepackage[utf8]{inputenc}

\usepackage[colorlinks, citecolor=blue, linkcolor=blue]{hyperref}
\usepackage{amsmath, amsfonts, amssymb, amsthm, graphicx, enumerate}
\usepackage[ruled]{algorithm2e}
\usepackage{subcaption}
\usepackage{orcidlink}
\usepackage{tikz-network}

\usepackage{setspace}
\newcommand\hl[1]{#1}
\newcommand{\hlmath}[1]{#1}

\usepackage{natbib}
\newtheorem{theorem}{Theorem}
\newtheorem{proposition}[theorem]{Proposition}
\newtheorem{lemma}[theorem]{Lemma}
\newtheorem{remark}[theorem]{Remark}
\newtheorem{corollary}[theorem]{Corollary}
\newtheorem{assumption}[theorem]{Assumption}

\newtheorem{example}[theorem]{Example}

\newcommand{\norm}[1]{\left\lVert#1\right\rVert}
\newcommand{\ip}[1]{\left\langle #1 \right\rangle}
\newcommand{\W}{\mathcal{W}}
\newcommand{\e}{\epsilon}

\newcommand{\Z}{\mathbb{Z}}

\newcommand{\E}{\mathbb{E}}
\newcommand{\Prob}{\mathbb{P}}
\renewcommand{\P}{\mathcal{P}}
\newcommand{\N}{\mathcal{N}}
\newcommand{\Q}{\mathcal{Q}}
\newcommand{\A}{\mathcal{A}}

\renewcommand{\H}{\mathcal{H}}

\newcommand{\X}{\mathcal{X}}
\newcommand{\Y}{\mathcal{Y}}
\newcommand{\B}{\mathcal{B}}

\newcommand{\C}{\mathcal{C}}
\newcommand{\R}{\mathbb{R}}
\newcommand{\tr}{\text{tr}}

\let\temp\phi
\let\phi\varphi
\let\varphi\temp

\begin{document}

\title{Weak convergence of adaptive Markov chain Monte Carlo}
\author{Austin Brown \orcidlink{0000-0003-1576-8381} and Jeffrey S. Rosenthal \orcidlink{0000-0002-5118-6808}}
\maketitle

\begin{abstract}
%In certain situations, the total variation distance is inadequate to develop a convergence theory for adaptive Markov chain Monte Carlo algorithms.
%In particular, the theory may be applied to adapt reducible Markov chains.
This article develops general conditions for weak convergence of adaptive Markov chain Monte Carlo processes and is shown to imply a weak law of large numbers for bounded Lipschitz continuous functions. This allows an estimation theory for adaptive Markov chain Monte Carlo where previously developed theory in total variation may fail or be difficult to establish. Extensions of weak convergence to general Wasserstein distances are established along with a weak law of large numbers for possibly unbounded Lipschitz functions. Applications are applied to auto-regressive processes in various settings, unadjusted Langevin processes, and adaptive Metropolis-Hastings.
\end{abstract}



\section{Introduction}

Markov chain Monte Carlo (MCMC) provides a mean to estimate integrals with respect to a target probability measure from the empirical average of a Markov chain.
Many Markov chains require a delicate choice of tuning parameters to explore the state space properly such as Metropolis-Hastings \citep{roberts:rosenthal:2001} and discretized Langevin diffusions \citep{roberts:1996:langevin, robe:rose:1998:optimal}.
The optimal tuning parameter choice often depends on properties of the target probability measure which may be challenging to compute precisely.
At the same time, a poor tuning parameter choice may lead to unreliable estimation and diagnostics from the Markov chain. This motivates adaptive MCMC processes which automatically learn or adapt the tuning parameters of the Markov chain as the process progresses in time \citep{haario:2001, roberts:rosenthal:2007}.

The general theory for adaptive MCMC processes is accomplished through a convergence guarantee on the non-adapted Markov chain in the total variation distance combined with diminishing conditions on the adaptation of the tuning parameters \citep{roberts:rosenthal:2007}.
Numerous adaptation strategies are possible such as stochastic approximation \citep{robbins:monro:1951} or specifically designed strategies to rapidly decrease the adaptation as time progresses \citep{chimisov:2018}. 
The existing general theory results in the ability to approximate arbitrary bounded functions through a weak law of large numbers.
However, there has been increasing evidence that convergence in total variation is inadequate for many high dimensional target probability measures compared to convergence in Wasserstein distances from optimal transportation \citep{hairer:etal:2011, durmus:moulines:2015, qin:hobert:2021, qin:hobert:2022}.
The issues with analyzing convergence with total variation are not limited to high dimensions and may appear for certain diffusion processes in any dimension \citep{hairer:etal:2011} and even toy examples \citep{tweedie:1977, butkovsky:2014}. 

Since the introduction of adaptive MCMC \citep{haario:2001}, many advancements have been made based upon convergence in total variation \citep{yves:2005, haario:2005, roberts:rosenthal:2009} but weak convergence appears less explored.
For example, convergence theory for adaptive MCMC has been extended to handle augmented target distributions that may depend on the adaptation to target multi-modal distributions \citep{pompe:2020}. 
%Strong conditions on diminishing adaptation have been shown to yield a weak law of large numbers for unbounded functions \citep{yang:rosenthal:2009}.
Under specific adaptation strategies based on stochastic approximation, convergence theory under stronger assumptions can lead to a central limit theorem \citep{andrieu:2006}.
However, each of these theoretical results and guarantees are based on convergence of the non-adapted Markov chain in total variation.
%Current theoretical results for adaptive MCMC based on total variation can be difficult to establish or using the total variation distance or are not well-suited for sampling high dimensional probability measures where the total variation metric often degenerates. 

This article's main contribution is the weak convergence of adaptive MCMC processes under general conditions using Wasserstein distances that metrize the weak convergence of probability measures \citep{Villani2008, gibbs:2004}.
%This is extended to convergence in general Wasserstein distances.
%This is accomplished first through convergence in a   Wasserstein distance which metrizes weak convergence of probability measures.
Section~\ref{section:adaptive_mcmc} introduces the general adaptive MCMC regime, and Section~\ref{section:wasserstein_distances} reviews the existing theory and some motivating examples that emphasize the inadequacy of the existing convergence theory.
Section~\ref{section:convergence} extends the traditional convergence framework in total variation for adaptive MCMC \citep{roberts:rosenthal:2007} to a framework based on weak convergence.
While the convergence result is weaker than total variation, it provides theoretical guarantees for approximations of bounded Lipschitz functions and arbitrary closed sets via Strassen's theorem \citep{strassen:1965}.
Section~\ref{section:lln} develops general conditions for a weak law of large numbers applied to bounded Lipschitz functions based on weak convergence.

\hl{Some examples and applications are explored in Section~\ref{section:convergence} with adapted auto-regressive processes, adaptive unadjusted Langevin processes, adaptive Langevin diffusions, and adaptive Metropolis-Hastings.
Beyond the examples studied here, the weak convergence theory for adaptive MCMC can be used to develop new adaptive algorithms for Bayesian inverse problems popular in physics that involve sampling posterior distributions on infinite-dimensional spaces where total variation can be problematic \citep{cotter:etal:2013}.
Another potentially useful application of the theory developed here is demonstrated in the adaptive Langevin diffusion example where using Wasserstein distances to show weak convergence can yield simpler proofs of the required conditions in comparison to proofs needed to show convergence in total variation.}

Weak convergence and the law of large numbers are further extended to general Wasserstein distances under stronger conditions.
The main application of this extension is a law of large numbers for unbounded Lipschitz functions in Section~\ref{section:lln} and is of practical relevance in statistics.
In particular, this extends the weak law of large numbers for Lipschitz functions for adaptive MCMC processes \citep{roberts:rosenthal:2007} and for Markov chains \citep{sandric:2017}.
Recently, a law of large numbers for bounded Lipschitz functions has been developed under strong contraction conditions in Wasserstein distances combined with strong limitations on the adaptation \citep{hofstadler:2024}.
The law of large numbers developed here holds under more general conditions and can apply to unbounded Lipschitz functions under suitable conditions.
Section~\ref{section:conclusion} discusses our theoretical results along with limitations and potential extensions of the newly developed theory.

\section{Background: adaptive Markov chain Monte Carlo processes}
\label{section:adaptive_mcmc}

Let $\Z_+$ denote the positive integers and denote the minimum and maximum of $a, b \in \R$ by $a \wedge b$ and $a \vee b$ respectively.
For a Borel measurable space $S$, let $\B(S)$ denote its Borel sigma field.
The Euclidean norm is denoted by $\norm{\cdot}$ and for a measure $\mu$, Lebesgue spaces are denoted by $L^p(\mu)$.
For a real-valued function $f$, denote the optimal Lipschitz constant with respect to a metric $d$ by $\norm{f}_{\text{Lip}(d)} = \sup_{ x \not= y} |f(y) - f(x)| / d(x, y)$.

We follow closely to the adaptive MCMC process framework of \citep{roberts:rosenthal:2007}.
Let $( \Omega, \B(\Omega) )$ be a Borel measurable space and let $\X$ and $\Y$ be complete separable metric spaces with respect to some metrics and $\B(\X)$ and $\B(\Y)$ be their respective Borel sigma fields. 
Let $\pi$ be a target Borel probability measure on $\X$.
For a discrete time index $t \in \Z_+$, the adaptive process updates a random tuning parameter $\Gamma_t : \Omega \mapsto \Y$ as the process progresses using the entire history to improve the distribution of $X_t : \Omega \mapsto \X$. 
The result is for the marginal distribution of $X_t$ to approximate the target distribution $\pi$.

We define generalized Borel measurable probability transition kernels $(\Q_{t})_{t \ge 0}$ with $\Q_t : (\Y \times \X)^t \times \B(\Y) \mapsto [0, 1]$ and a family of Borel measurable Markov transition kernels $(\P_{\gamma})_{\gamma \in \Y}$ with $\P_\gamma : \X \times \B(\X) \mapsto [0, 1]$ to prescribe the adaptive process by the relations
\begin{align*}
&\Prob(\Gamma_t \in d\gamma | H_{t-1} )
= \Q_{t}(H_{t-1}, d\gamma)
&\Prob(X_t \in dx | \Gamma_t, X_s, H_{t-1} )
= \P_{\Gamma_t}(X_s, dx)
\end{align*}
where $H_{t} = (\Gamma_0, X_0, \ldots, \Gamma_{t}, X_{t})$ denotes the history at time $t$.
This prescribes the finite-dimensional distributions so that \hl{$(\Gamma_0, X_0, \ldots, \Gamma_t, X_t)$ for fixed $\Gamma_0, X_0$ has joint distribution}
\begin{align*}
\hlmath{ \A^{(0, \ldots, t)}((\gamma_0, x_0), d\gamma_1, dx_1, \ldots, d\gamma_t, dx_t)
= \prod_{k = 1}^{t} \Q_{k}(h_{k - 1}, d\gamma_{k}) \P_{\gamma_{k}}(x_{k - 1}, dx_{k}) }
\end{align*}
with history $\hlmath{h_{k} = ( \gamma_{0}, x_{0}, \ldots, \gamma_{k}, x_{k} )}$.
This defines an \textit{adaptive process} $(\Gamma_t, X_t)_{t \ge 0}$ adapted to the filtration $\H_t = \B(\Gamma_s, X_s : 0 \le s \le t)$ and initialized at any probability measure $\mu$ on $( \Y \times \X, \B(\Y \times \X) )$ by the Ionescu-Tulcea extension theorem \citep{tulcea:1949}.

We will mostly be concerned with the marginal distribution $X_t$ from fixed initialization points $\gamma, x \in \Y \times \X$ and general initializations $\mu$ on $( \Y \times \X, \B(\Y \times \X) )$, defined by
\begin{align}
&X_t | \Gamma_0, X_0 = \gamma, x \sim \A^{(t)}((\gamma, x), \cdot),
&X_t \sim \mu\A^{(t)}(\cdot)
= \int_{\Y \times \X} \A^{(t)}((\gamma, x), \cdot) \mu(d\gamma, dx)
\label{eq:adaptive_process}
\end{align}

\section{Background: Wasserstein distances}
\label{section:wasserstein_distances}

Let $d : \X \times \X \to [0, \infty)$ be a lower-semi-continuous metric. Define the Wasserstein distance or transportation distance of order $p \in \Z_+$ between two arbitrary Borel probability measures $\mu$ and $\nu$ on $(\X, \B(\X))$ by
\[
\W_{d, p}(\mu, \nu)
= \left( \inf_{\xi \in \C(\mu, \nu)} \int_{\X \times \X} d(x, y)^p \xi(dx, dy) \right)^{1/p}
\]
where $\C(\mu, \nu)$ is the set of all joint probability measures $\xi$ such that $\xi(\cdot \times \X) = \mu(\cdot)$ and $\xi(\X \times \cdot) = \nu(\cdot)$.
%Wasserstein distances $\W_{c, p}(\cdot, \cdot)$ are distances when $c(\cdot, \cdot)$ is a metric on $\X$ and often called Wasserstein distances \citep[Chapter 6]{Villani2008}.
We generally suppress the $1$ in the $L^1$ metric and denote $\W_{d}(\mu, \nu) := \W_{d, 1}(\mu, \nu)$.

The total variation distance denoted by $\W_{\text{TV}}(\cdot, \cdot)$ between probability measures can be seen as a special case of a Wasserstein distance when the metric is defined by $I_{D^c}(\cdot, \cdot)$ with the off-diagonal set $D^c = \{ x, y \in \X \times \X : x \not= y \}$.
If $(\X, d)$ is a complete separable metric space, then the standard bounded metric $d(\cdot, \cdot) \wedge 1$ defines a Wasserstein distance $\W_{d \wedge 1}(\cdot, \cdot)$. This Wasserstein distance metrizes the weak convergence of probability measures through the bounded Lipschitz metric \citep[Theorem 11.3.3]{dudley:2018} and is equivalent up to a constant to the bounded Lipschitz metric by Kantovorich-Rubinstein duality \citep[Theorem 1.4]{Villani2003}.

Traditional theory of adaptive MCMC considers an adaptive process $(\Gamma_t, X_t)_{t \ge 0}$ initialized at $\Gamma_0, X_0 = \gamma, x \in \Y \times \X$ satisfying a \textit{(strong) diminishing adaptation} condition
\begin{align}
&\lim_{t \to \infty} \sup_{x \in \X} \W_{\text{TV}}(\P_{\Gamma_{t + 1}}(x, \cdot), \P_{\Gamma_{t}}(x, \cdot))
= 0
\label{assumption:strong_da}
\end{align}
\hl{in probability} with the supremum assumed Borel measurable and a \textit{(strong) containment} condition which is to show for any $\e \in (0, 1)$, the sequence
\begin{align}
M_{\e}(\Gamma_t, X_t)
= \inf\left\{ N \in \Z_+ : 
\W_{\text{TV}}\left( \P_{\Gamma_t}^n(X_t, \cdot), \pi \right) \le \e
\text{ for all } n \ge N \right\}
\label{assumption:strong_containment}
\end{align}
is bounded in probability, that is, 
$
\lim_{N \to \infty} \sup_{t \ge 0} \Prob\left( M_{\e}(\Gamma_t, X_t) > N \right) = 0.
$
The (strong) diminishing adaptation restricts the adaptation plan for $(\Gamma_t)_t$ and (strong) containment is a uniform convergence requirement on the non-adapted Markov chain.

Under these two conditions, one has the guarantee \citep{roberts:rosenthal:2007} that for every fixed initialization $\gamma, x \in \Y \times \X$ and for every bounded Borel measurable function $\phi : \X \to \R$,
\begin{align*}
&\lim_{t \to \infty} \W_{\text{TV}}\left( \A^{(t)}( (\gamma, x), \cdot), \pi \right)
= 0
&\text{ and }
&&\frac{1}{t} \sum_{s = 1}^t \phi(X_s)
\to \int_\X \phi d\pi
\end{align*}
in probability as $t \to \infty$.
Both of these guarantees have many practical applications in the reliability of Monte Carlo simulations in Bayesian statistics.
General conditions for (strong) containment \eqref{assumption:strong_containment} to hold have also been developed \citep{bai:2011, krys:rosenthal:2014}.

The (strong) containment condition \eqref{assumption:strong_containment} is often established via simultaneous drift and minorization conditions \citep{roberts:rosenthal:2007, bai:2011}.
This requires a drift function $V : \X \to [0, \infty)$ and identification of a small set $S = \{ x \in \X : V(x) \le R \}$ so that there is are constants $\lambda, \alpha \in (0, 1)$ and $L \in (0, \infty)$ where is it required $R > 2 L / (1 - \lambda)$ and for every $\gamma, x \in \Y \times \X$, there is a Borel probability measure $\nu_\gamma$ on $\X$ such that
\begin{align}
&\inf_{y \in S} \P_{\gamma}(y, \cdot) 
\ge \alpha \nu_\gamma(\cdot)
&\text{ and }
&&(\P_\gamma V) (x) \le \lambda V(x) + L.
\label{eq:dm}
\end{align}
These techniques yield (strong) containment \eqref{assumption:strong_containment} through a geometric rate $r \in (0, 1)$ \hl{and constant $M_0 > 0$} so that for $t, n \in \Z_+$
\begin{align*}
\W_{\text{TV}}\left( \P_{\Gamma_t}^n(X_t, \cdot), \pi \right)
\le \hlmath{M_0} r^n V(X_t)
\end{align*}
and $( V(X_t) )_{t \ge 0}$ is bounded in probability \citep[Theorem 18]{roberts:rosenthal:2007}.

The identification of such a small set $S$ and drift function $V$ as in \eqref{eq:dm} often becomes problematic in large dimensions as probability measures often tend towards mutual singularity. Even in low dimensions, a small set may not exist as a non-adapted Markov kernel may fail to be irreducible meaning for each $\gamma \in \Y$, there is no Borel probability measure $\varphi_\gamma$ on $\X$ such that $\varphi_\gamma(\cdot) > 0$ implies $\P_\gamma(x, \cdot) > 0$ for all $x \in \X$.
In this case, it is not possible to find such a small set regardless of the dimension \citep[Theorem 5.2.2]{meyn:tweedie:2012}.

\section{Motivating examples}

The following running examples illustrate problematic points with analysis in total variation for adapting the tuning parameters of Markov chains compared to their alternative weak convergence properties.
In particular, (strong) containment \eqref{assumption:strong_containment} may fail.
Stark differences in the convergence characteristics may even appear when adapting a discrete Markov chain as the following example illustrates.

\begin{example}
\label{example:discrete_ar_process}
(Discrete auto-regressive process)
Let $\gamma \in \Z_+$ with $\gamma \ge 2$ and $(\xi^{\gamma}_t)_{t \ge 0}$ be independent uniformly distributed discrete random variables on $\{ 0, 1/\gamma, 2/\gamma, \ldots, (\gamma - 1)/\gamma \}$.
With $X_0 = x \in [0, 1)$, define the auto-regressive process for $t \in \Z_+$ by
\[
X^{\gamma}_{t} = \frac{1}{\gamma} X^{\gamma}_{t-1} + \xi^{\gamma}_{t}.
\]
For each fixed $\gamma \ge 2$, this defines a Markov chain with Markov transition kernel denoted by $\P_\gamma$. 
It can be shown the invariant probability measure $\pi \equiv \text{Unif}(0, 1)$ is Lebesgue measure on $[0, 1)$.

For any adaptive process $(\Gamma_t, X_t)_{t \ge 0}$ using these Markov kernels $(\P_\gamma)_{\gamma}$, traditional convergence theory in total variation \citep[Theorem 13]{roberts:rosenthal:2007} is inadequate.
Indeed, it can be shown that
\[
\W_{\text{TV}}\left( \P^t_\gamma(x, \cdot), \text{Unif}(0, 1) \right) = 1
\]
and (strong) containment \eqref{assumption:strong_containment} fails under any adaptive strategy.
On the other hand, weak convergence is exponentially fast.
For $t \in \Z_+$ and any fixed $\gamma$, starting from $X_0 = x \in [0, 1)$ and $Y_0 = y \in [0, 1)$, define $X^{\gamma}_{t} = \frac{1}{\gamma} X^{\gamma}_{t-1} + \xi^{\gamma}_{t}$ and $Y^{\gamma}_{t} = \frac{1}{\gamma} Y^{\gamma}_{t-1} + \xi^{\gamma}_{t}$ with shared discrete uniformly distributed random variable $\xi^{\gamma}_{t}$. 
These random variables $(X^{\gamma}_{t}, Y^{\gamma}_{t})$ define a coupling so that for any $x \in [0, 1)$ and $\gamma \ge 2$
\begin{align*}
\W_{ |\cdot|}\left( \P^t_\gamma(x, \cdot), \text{Unif}(0, 1) \right) 
%&\le \int_{[0, 1)} \W_{ |\cdot|}\left( \P^t_\gamma(x, \cdot), \P^t_\gamma(y, \cdot) \right) dy
%\\
&\le \int_{[0, 1)} \E\left[ \left| X^{\gamma}_{t} - Y^{\gamma}_{t}\right| X_0 = x, Y_0 = y \right] dy
\\
&\le 2^{-t}.
\end{align*}
In particular, we will show later that under a suitable adaptation strategy, the adaptive process converges weakly using this Wasserstein distance.
\end{example}

The next example shows how problems appear in infinite-dimensions. Although the example is somewhat abstract, poor scaling properties in infinite-dimensions can also appear in practical high-dimensional scenarios in statistics.

\begin{example}
\label{example:infinite_ar_process}
(Infinite-dimensional auto-regressive process)
Consider a separable Hilbert space $H$ with infinite dimension and inner product $\ip{\cdot, \cdot}$.
Let $\N(0, C)$ be a Gaussian Borel probability measure on $H$ with mean $0 \in H$ and symmetric positive covariance operator $C : H \to H$ such that
$
\tr(C)  = \sum_{k = 1}^{\infty} \langle C u_k, u_k \rangle < \infty
$
where $(u_k)_k$ is any orthonormal basis of $H$.
We will further assume $C$ is non-degenerate so that for every $x, y \in H$, $C x \equiv 0 \in H$ implies $x \equiv 0 \in H$.
For some $\gamma^* < 1$, consider the family of Markov transition kernels $(\P_{\gamma})_{\gamma\in (0, \gamma^*)}$ for the auto-regressive process $(X^{\gamma}_t)_{t \ge 0}$ where $(\xi_{t})_t$ are independent with $\xi_{t} \sim \N(0, C)$ and
\begin{align}
&X^{\gamma}_t = \gamma X^{\gamma}_{t-1} + \sqrt{1 - \gamma^2} \xi_{t},
&t \in \Z_+.
\label{eq:ar_infinite}
\end{align}
For any fixed $\gamma \in (0, \gamma^*)$, if $X^{\gamma}_{t-1} \sim \N(0, C)$, then $X^{\gamma}_{t} \sim \N(0, C)$ and the invariant probability measure is $\N(0, C)$.

For an adaptive auto-regressive process $(\Gamma_t, X_t)_{t \ge 0}$ defined by \eqref{eq:ar_infinite}, convergence theory in total variation \citep[Theorem 13]{roberts:rosenthal:2007} fails to provide a convergence guarantee.
For each $x \in H$ and $\gamma \in (0, 1)$,
\[
\W_{\text{TV}}( \P_{\gamma}(x, \cdot), \N(0, C) ) = 1
\]
due to the covariances differing and the Feldman-Hajeck Theorem (see \citep[Theorem 2.25]{prato:etal:2014}).
It follows that (strong) containment \eqref{assumption:strong_containment} cannot hold under any adaptation strategy \eqref{assumption:strong_da}.
However, convergence in $L^2$-Wasserstein distances is exponentially fast.
Initialized with $X_0 = x \in H$ and $Y_0 = y \in H$, define 
\begin{align*}
&X^{\gamma}_t = \gamma X^{\gamma}_{t-1} + \sqrt{1 - \gamma^2} \xi_{t}
&\text{ and }
&&Y^{\gamma}_t = \gamma Y^{\gamma}_{t-1} + \sqrt{1 - \gamma^2} \xi_{t}
\end{align*}
using common random variable $\xi_t \sim N(0, C)$.
This defines a coupling so that the $L^2$-Wasserstein distance is upper bounded with
\begin{align*}
\W_{\norm{\cdot}, 2}( \P_{\gamma}^t(x, \cdot),  \N(0, C) ) 
&\le \left[  \int_H \W_{\norm{\cdot}, 2}( \P_{\gamma}^t(y, \cdot),  \P_{\gamma}^t(x, \cdot) )^2 \pi(dy) \right]^{1/2}
%\\
%&\le \left[ \int_H \E\left[ \norm{X^\gamma_t - Y^\gamma_t}^2 \bigm| X_0 = x, Y_0 = y \right] \pi(dy) \right]^{1/2}
\\
&\le {\gamma^*}^t \left[ \norm{x} + \sqrt{\tr(C)} \right].
\end{align*} 
\end{example}

\section{Main results}
\label{section:convergence}

This section extends previous results on convergence in total variation of adaptive MCMC processes to weak convergence and general Wasserstein distances \citep[Theorem 5]{roberts:rosenthal:2007} and \citep[Theorem 13]{roberts:rosenthal:2007}.
Let $\rho(\cdot, \cdot)$ be a lower semicontinuous metric on $\X$ so $\W_{\rho \wedge 1}(\cdot, \cdot)$ defines a Wasserstein distance.
If $(X, \rho)$, is a complete separable metric space, then $\W_{\rho \wedge 1}(\cdot, \cdot)$ metrizes weak convergence \citep[Theorem 11.3.3]{dudley:2018}.
A motivation for this convergence is Strassen's theorem which gives approximations to arbitrary closed sets \citep[Corollary 1.28]{Villani2003}.
However, $\rho(\cdot, \cdot)$ need only satisfy the axioms of a metric and $\W_{\rho \wedge 1}(\cdot, \cdot)$ is defined more generally.

The first simple situation is to introduce a stopping time $T$ such that the adaptation terminates and $\Gamma_T = \Gamma_t$ for all $t \ge T$.
For any $T \ge 1$ determining a stopping point of adaptation, we can construct a finite adaptation process $(Y_t, \Gamma_t)_{t = 0}^{\infty}$ adapted to the filtration $\H'_t = \B(Y_s, \Gamma_s : 0 \le s \le t)$ initialized at $\Gamma_0, Y_0, = \gamma, x$ such that $Y_t = X_t$ for $0 \le t \le T$ and is a Markov chain further out, that is, $Y_{t + 1} | \H'_{t}, Y_{t} = y \sim \P_{\Gamma_{T}}(y, \cdot)$ for $t \ge T$.
Denote the marginal distribution $B^{(T, t - T)}((\gamma, x), \cdot) = \Prob(Y_t \in \cdot | \Gamma_0, Y_0 = \gamma, x)$.

\begin{proposition}
Let $\rho(\cdot, \cdot)$ be a lower semicontinuous metric on $\X$ and let $(\Gamma_t, X_t)_{t \ge 0}$ a finite adaptive process with initialization probability measure $\mu$ as in \eqref{eq:adaptive_process}.
If for every initialization $x \in \X$ and every $\gamma \in \Y$,
$
\lim_{t \to \infty} \W_{\rho \wedge 1}(\P_{\gamma}^t(x, \cdot), \pi) = 0,
$
then
\begin{align*}
\lim_{t \to \infty} \W_{\rho \wedge 1}(\mu B^{(T, t - T)}, \pi) = 0.
\end{align*}
\end{proposition}

\begin{proof}
Since the optimal coupling exists \citep[Theorem 4.1]{Villani2008} and is Borel measurable \citep[Corollary 5.22]{Villani2008},
\[
\W_{\rho \wedge 1}(\mu B^{(T, t - T)}, \pi)
\le \E\left[ \W_{\rho \wedge 1}(\P_{\Gamma_{T}}^{t - T}(X_T, \cdot), \pi) \right].
\] 
The conclusion follows by dominated convergence.
\end{proof}

While finite adaptation may be a safe strategy, infinite adaptation where the process continually learns tuning parameters is often of greater interest in applications \citep{roberts:rosenthal:2007}.
Consider now the following two weakened assumptions both generalized from \citep{roberts:rosenthal:2007}.
The first assumption is a weak restriction on the adaptation and the second is a weak containment condition on convergence of the non-adapted Markov chain.

\begin{assumption}
Let $(\Gamma_t, X_t)_{t \ge 0}$ an adaptive process with initialization probability measure $\mu$ \hl{as in} \eqref{eq:adaptive_process}.
Let $\rho(\cdot, \cdot)$ and  $\tilde{\rho}(\cdot, \cdot)$ be a lower semicontinuous metrics on $\X$ such that $\rho(\cdot, \cdot) \wedge 1 \le \tilde{\rho}(\cdot, \cdot) \wedge 1$.
Suppose the following conditions hold:

\begin{enumerate}
\item (Weak containment)
Suppose for any $\e \in (0, 1)$, the sequence 
\[
\hlmath{M_{\e, \rho}(\Gamma_t, X_t)}
:= \inf\left\{ N \ge 0 : \W_{\rho \wedge 1}\left[ \P_{\Gamma_t}^n(X_t, \cdot), \hlmath{\pi} \right] \le \e \text{ for all $n \ge N$} \right\}
\]
is bounded in probability, that is,
\begin{align}
\lim_{T \to \infty} \sup_{t \ge 0} \Prob\left( \hlmath{M_{\e, \rho}(\Gamma_t, X_t)} \ge T \right) = 0.
\label{assumption:weak_containment}
\end{align}

\item (Weak diminishing adaptation)
Suppose there is a conditional coupling $x, y, \gamma, \gamma' \mapsto \xi_{x, y, \gamma', \gamma} \in \C[\P_{\gamma'}(x, \cdot), \P_{\gamma}(y, \cdot)]$ and a real-valued sequence $(\delta_k)_{k \ge 0}$ with $\lim_{k \to \infty} \delta_k = 0$ such that
\begin{align}
D_{\tilde{\rho}}(\Gamma_{t + 1}, \Gamma_t)
:=
\lim_{k \to \infty}  \sup_{ \{ x, y \in \X \times \X : \tilde\rho(x, y) \le \delta_k \} } 
\int_{\X \times \X} \tilde\rho(x', y') \wedge 1 \xi_{x, y, \Gamma_{t + 1}, \Gamma_{t}}(dx', dy')
%\W_{\tilde\rho \wedge 1}\left( \P_{\Gamma_{t + 1}}(x, \cdot), \P_{\Gamma_{t}}(y, \cdot) \right) 
\label{assumption:weak_da}
\end{align}
is $\H_{t + 1}$ measurable and $D_{\tilde{\rho}}(\Gamma_{t + 1}, \Gamma_t) \to 0$ in probability as $t \to \infty$.

\end{enumerate}

\end{assumption}

There are existing results to bound the convergence rate of non-adapted Markov chains which can be modified to satisfy the weak containment condition using drift and coupling techniques \citep{hairer:etal:2011,durmus:moulines:2015}.
Note that $\rho(\cdot, \cdot) \wedge 1 \le I_{ \{ x \not= y \} }(\cdot, \cdot)$ and (strong) diminishing adaptation \citep{roberts:rosenthal:2007} implies weak diminishing adaptation \eqref{assumption:weak_da}. 
We then immediately have Proposition~\ref{proposition:implications} below.
In certain cases it may be simpler to show (strong) diminishing adaptation where only weak containment \eqref{assumption:weak_containment} holds and the implications of Proposition~\ref{proposition:implications} below are visualized in Figure~\ref{figure:conditions}.

\begin{proposition}
\label{proposition:implications}
Let $(\Gamma_t, X_t)_{t \ge 0}$ be an adaptive process with initialization probability measure $\mu$ \hl{as in} \eqref{eq:adaptive_process}.
If the process satisfies (strong) containment \eqref{assumption:strong_containment}, then weak containment \eqref{assumption:weak_containment} is satisfied. 
If the process satisfies (strong) diminishing adaptation \eqref{assumption:strong_da}, then weak diminishing adaptation \eqref{assumption:weak_da} is satisfied.
\end{proposition}

\begin{figure}
\begin{center}
\includegraphics{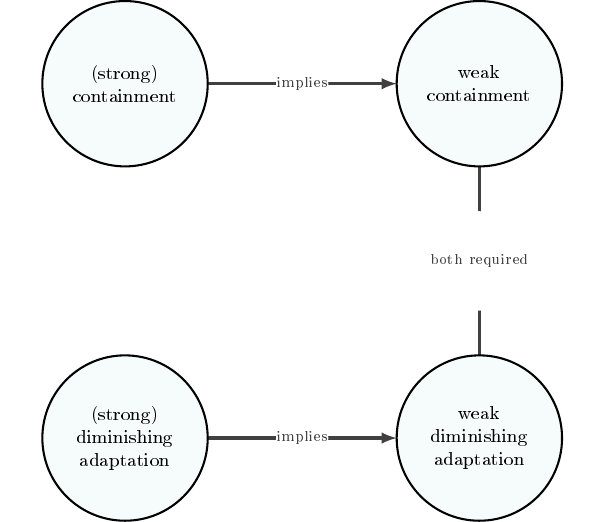}
\end{center}
\caption{Graph illustration for comparison of strong/weak containment and strong/weak diminishing adaptation conditions required to obtain weak convergence of adaptive MCMC.}
\label{figure:conditions}
\end{figure}

The following result shows weak convergence of the adaptive MCMC process.

\begin{theorem}
\label{theorem:convergence}
Let $(\Gamma_t, X_t)_{t \ge 0}$ be an adaptive process with initialization probability measure $\mu$ \hl{as in} \eqref{eq:adaptive_process}.
If weak containment \eqref{assumption:weak_containment} holds and weak diminishing adaptation \eqref{assumption:weak_da} holds, then
\[
\lim_{t\to\infty} \W_{\rho \wedge 1}( \mu\A^{(t)}, \pi )
= 0.
\]
\end{theorem}

We will prove Thereom~\ref{theorem:convergence} through subsequent lemmas by comparing the adaptive process to an adaptive process where adaptation stops at a finite time.
The first result shows that weak containment ensures the convergence of the finite adaptation process to the target measure uniformly in the finite adaptation stopping time.

\begin{lemma}
\label{lemma:convergence_wc}
If weak containment holds \eqref{assumption:weak_containment}, then for any $\gamma, x \in \Y \times \X$,
\[
\lim_{n \to \infty} \sup_{T \ge n} \W_{\rho \wedge 1}\left[ B^{(T, n - T)}( (\gamma, x), \cdot), \pi \right]
= 0.
\]
\end{lemma}

\begin{proof}
Fix $\e \in(0, 1)$.
For any $\gamma, x \in \Y \times \X$ and each $n \in \Z_+$, the infimum is attained at an optimal coupling $\xi_{x, \gamma}^{(n)} \in \C[\P_{\gamma}^n(x, \cdot), \pi]$ \citep[Theorem 4.1]{Villani2008} so that
\begin{align}
\W_{\rho \wedge 1}\left[\P_{\gamma}^n(x, \cdot), \pi \right]
= \int_{\X^2} \rho(x', y') \wedge 1 \xi^{(n)}_{\gamma, x}(dx', dy').
\end{align}
The coupling is Borel measurable due to $\rho(\cdot, \cdot) \wedge 1$ being lower semi-continuous and can be approximated by a non-decreasing sequence of bounded Lipschitz functions so we can choose a measurable selection \citep[Corollary 5.22]{Villani2008} such that the limit is Borel measurable using approximation techniques in \citep[Theorem 1.3]{Villani2003}.
Define the set $A_\e = \{ \gamma, x \in \Y \times \X : M_{\e, \rho}(\gamma, x) \le N \}$.
For all $\gamma, x \in A_\e$ and for all $n \ge N$,
\begin{align}
\W_{\rho \wedge 1}\left[\P_{\gamma}^n(x, \cdot), \pi \right]
&\le \e.
\label{eq:weak_containment_bound}
\end{align}

Let $\nu_{\gamma, x}^{(T)}$ denote the probability measure for $( X_{T}, \Gamma_{T} )$ given $\Gamma_0, X_0 = \gamma, x$.
Then
\[
\hat{\xi}^{(T + n)}_{\gamma, x}(dx_{T + n}, dy) = 
\int_{\Y \times \X} \xi_{\gamma_{T}, y_{T}}^{(n)}(dx', dy') \nu^{(T)}_{\gamma, x}(\gamma_{T}, y_{T})
\]
defines a coupling for the finite adaptation process $Y_{T + n} \sim B^{(T, n)}((\gamma, x), \cdot)$ and $Y \sim \pi$ \citep[Theorem 4.8]{Villani2008}.
By the weak containment assumption \eqref{assumption:weak_containment}, there is an $N$ depending on $\e$ such that uniformly in $T \ge 0$, 
$
\nu^{(T)}_{\gamma, x}(A_{\e}^c)
= \Prob\left( M_{\e, \rho}(\Gamma_T, X_T) > N \right) \le \e.
$
Using \eqref{eq:weak_containment_bound}, then uniformly in $T \ge n$,
\begin{align*}
&\W_{\rho \wedge 1}\left[ B^{(T, n - T)}( (\gamma, x), \cdot), \pi \right]
\le \int_{\X^2} \rho(x', y') \wedge 1 \hat\xi_{\gamma, x}^{(T + n)}(dx', dy')
\\
&\le
\int_{\X^2} \int_{A_\e} \rho(y_{T + n}, y) \wedge 1  \xi_{\gamma_{T}, y_{T}}^{(n)}(dx', dy') \nu^{(T)}_{\gamma, x}(\gamma_{T}, y_{T})
+ \sup_{T' \ge 0} \nu^{(T')}_{\gamma, x}(A_{\e}^c)
\\
&\le 2 \e.
\end{align*}
\end{proof}

The weak diminishing adaptation condition will ensure our next goal which is to have the adaptive MCMC process converge to the finite adaptation process.

\begin{lemma}
\label{lemma:convergence_weak_da}
If weak diminishing adaptation \eqref{assumption:weak_da} holds, then for any $\gamma, x \in \Y \times \X$ and any $N \ge 0$
\[
\lim_{T \to \infty} \W_{\tilde{\rho} \wedge 1}\left( \A^{(T + N)}( (\gamma, x), \cdot), B^{(T, N)}( (\gamma, x), \cdot) \right)
= 0.
\]
\end{lemma}

\begin{proof}
It will suffice to assume $\tilde\rho = \rho$ and the optimal coupling in the weak diminishing adaptation assumption \eqref{assumption:weak_da}.
Fix $N \ge 1$ and fix $\e \in(0, 1)$.
For each $\gamma, \gamma'$ and each $x, y$, there exists a Borel measurable optimal coupling $\xi_{x, y, \gamma', \gamma}^*$ so that
\[
\W_{\rho \wedge 1}\left[ 
\P_{\gamma'}(x, \cdot), \P_{\gamma}(y, \cdot)
\right]
= \int_{\X^2} \rho(x', y') \wedge 1 \xi_{x, y, \gamma', \gamma}^*(dx', dy').
\]
Using these conditional couplings, we define a joint probability measure $\zeta_{\gamma_0, x_0}$ by
\begin{align*}
&\zeta_{\gamma_0, x_0}(dx_{1}, d\gamma_1, dy_{1}, \ldots, dx_{T + N}, d\gamma_{T + N}, dy_{T + N} )
\\
&= \prod_{s = 1}^{T} \P_{\gamma_s}(x_{s-1}, dx_s) \Q_{s}(h_{s-1}, d\gamma_s) \hlmath{\delta_{x_1, \ldots, x_s}}(dy_1, \ldots, dy_{s})
\\
&\hspace{1.2em}\prod_{s = T + 1}^{T + N} \xi_{x_{s-1}, y_{s-1}, \gamma_s, \gamma_T}^*(dx_{s}, dy_{s}) \Q_{s}(h_{s-1}, d\gamma_s)
\end{align*}
where for $0 \le s \le t$, the history $h_s = (\gamma_0, x_0, \ldots, \gamma_s, x_s)$.
The marginal is a coupling $\zeta_{\gamma_0, x_0}(dx_{t}, dy_{t} )$ for the adaptive process $X_t | \Gamma_0, X_{0} = \gamma, x$ and the finite adaptation process $Y_t | \Gamma_0, Y_{0} = \gamma, x$ initialized so that they are identical up to time $T$ and use conditional couplings thereafter.

For $\gamma', \gamma \in \Y$ and $\delta \in (0, 1)$, define 
\[
D_{\rho, \delta}(\gamma', \gamma)
=
\sup_{ \{ x, y : \rho(x, y) \le \delta \} } 
\W_{\rho \wedge 1}\left( \P_{\gamma'}(x, \cdot), \P_{\gamma}(y, \cdot) \right).
\]
For any $\e', \delta' \in (0, 1)$ and $k \in \Z_+$, define the set 
 \[
 E^{(T, N)}_{\e', \delta'} = \left\{ \gamma_{T + 1}, \ldots, \gamma_{T + N} : 
D_{\rho, \delta'}(\gamma_{t + 1}, \gamma_{t})
\le \e'/N^2, T + 1 \le t \le T + N -1 \right\}.
\]
Starting with $\delta_N = r \in (0, 1)$, then for each $1 \le k \le N$, given $\delta_k \in (0, 1)$, by weak diminishing adaptation \eqref{assumption:weak_da}, we can choose $T$ large enough depending on $\e, N, \delta_k$ and we can choose $\delta_{k - 1}$ sufficiently small so that for all $t \ge T$, $D_{\rho, \delta}(\Gamma_{t + 1}, \Gamma_t)$ is $\H_{t + 1}$ measurable and
\[
\Prob\left( 
D_{\rho, \delta_{k - 1}}(\Gamma_{t + 1}, \Gamma_{t})
> \delta_{k}\e/N^2
\right)
\le \e/N^2.
\]
This constructs $\delta_0, \ldots, \delta_N = r$ so that using a union probability bound, we can choose $T$ sufficiently large depending on $\e, N, \delta_1, \ldots, \delta_N$ so that for each $1 \le k \le N$, we have
$
\Prob\left( E^{(T, N)}_{\e \delta_k, \delta_{k - 1}} \right)
\ge 1 - \e/N. 
$
Define $E = \bigcap_{k = 1}^N E^{(T, N)}_{\delta_k \e, \delta_{k - 1}}$ and a union probability bound then implies
\[
\Prob(E)
= \Prob\left( \bigcap_{k = 1}^N E^{(T, N)}_{\delta_k \e, \delta_{k - 1}} \right)
\ge 1 - \e. 
\]

The triangle inequality for the Wasserstein distance \citep[Lemma 7.6]{Villani2003} holds for every $1 \le k \le N$ with
\begin{align}
&\W_{\rho \wedge 1}\left[ 
\P_{\gamma_{T + k}}(x, \cdot), \P_{\gamma_{T}}(y, \cdot)
\right]
\\
&\le
\W_{\rho \wedge 1}\left[ 
\P_{\gamma_{T + 1}}(x, \cdot), \P_{\gamma_{T}}(y, \cdot)
\right]
+ \sum_{s = 1}^{N - 1} 
\W_{\rho \wedge 1}\left[ 
\P_{\gamma_{T + s + 1}}(x, \cdot), \P_{\gamma_{T} + s}(x, \cdot)
\right].
\label{eq:triangle_ub}
\end{align}
For each $k$, if $\gamma_{T + 1}, \ldots, \gamma_{T + k} \in E^{(T, N)}_{\e \delta_{k}, \delta_{k - 1} }$, then by \eqref{eq:triangle_ub} and Markov's inequality,
\[
\inf_{\{ \rho(x, y) \le \delta_{k - 1} \} }\W_{\{ \rho(x', y') \le \delta_{k} \}}\left[ 
\P_{\gamma_{T + k}}(x, \cdot), \P_{\gamma_{T}}(y, \cdot)
\right]
\ge 1 - \frac{\e}{N}.
\]

By construction of the the distribution $\zeta$, $\rho(x_T, y_T) \le \delta_0$ regardless of how small $\delta_0$ is and we have the lower bound
\begin{align*}
\zeta\left( \bigcap_{k = 1}^{N} \{ \rho(x_{T + k}, y_{T + k}) \le \delta_k \} \bigm| E \right)
&\ge \left( 1 - \frac{\e}{N} \right)^{N}
\\
&\ge 1 - \e.
\end{align*}
Combining these results, we have
\begin{align*}
&\W_{ I_{ \{ x', y' : \rho(x', y') > r \} } } \left( \A^{(T + N)}( (\gamma, x), \cdot), B^{(T, N)}( (\gamma, x), \cdot) \right)
\\
&\le \zeta\left(
\{ \rho(x_{T + N}, y_{T + N}) > r \} 
\right)
\\
&\le \zeta\left(\{ \rho(x_{T + N}, y_{T + N}) > \delta_N \} \bigm| E \right)
+ \Prob\left( E \right)
\\
&\le 2\e.
\end{align*}
Since this holds for any $r, \e$, 
\[
\W_{ \rho \wedge 1 } \left( \A^{(T + N)}( (\gamma, x), \cdot), B^{(T, N)}( (\gamma, x), \cdot) \right)
\le r + 2\e
\]
and the conclusion follows.
\end{proof}

Combining these lemmas, we may now prove Theorem~\ref{theorem:convergence}.

\begin{proof}[Proof of Theorem~\ref{theorem:convergence}]
Fix $\e \in(0, 1)$.
From Lemma~\ref{lemma:convergence_wc}, we can choose $N_\e$ sufficiently large so that for all $n \ge N_\e$ with a particular adaptation stopping time $T_n = n - N_\e \ge 0$,
\begin{align*}
\W_{\rho \wedge 1}\left[ B^{(T_n, N_\e)}( (\gamma, x), \cdot), \pi) \right]
\le \e/2.
\end{align*}
Given this $N_\e$ and using Lemma~\ref{lemma:convergence_weak_da}, we can choose $n_\e$ sufficiently large so that for all $n \ge n_\e$
\begin{align*}
\W_{\rho \wedge 1}\left[
\A^{(T_n + N_\e)}( (\gamma, x), \cdot), B^{(T_n, N_\e)}( (\gamma, x), \cdot) \right]
\le \e/2.
\end{align*}
The triangle inequality holds by \citep[Lemma 7.6]{Villani2003} so that
\begin{align*}
&\W_{\rho \wedge 1}\left[ \A^{(n)}( (\gamma, x), \cdot), \pi \right]
\\
&\le \W_{\rho \wedge 1}\left[
\A^{(T_n + N_\e)}( (\gamma, x), \cdot), B^{(T_n, N_\e)}( (\gamma, x), \cdot) \right]
+ \W_{\rho \wedge 1}\left[ B^{(T_n, N_\e)}( (\gamma, x), \cdot), \pi \right]
\\
&\le \e.
\end{align*}
Since the conditional optimal coupling is attained and is Borel measurable \citep[Theorem 4.8]{Villani2008}, we have by dominated convergence
\begin{align*}
\lim_{t \to \infty} \W_{\rho \wedge 1}\left[ 
\mu \A^{(n)}, \pi 
\right]
\le \lim_{t \to \infty} \int_{\Y \times \X} \W_{\rho \wedge 1}\left[ 
\A^{(n)}( (\gamma, x), \cdot), \pi 
\right] \mu(d\gamma, dx)
= 0.
\end{align*}
\end{proof}

Interestingly we do not assume that $\pi$ is invariant. 
Denote the distance to a closed set $C \subseteq \X$ by $\rho(x, C) = \inf_{y \in C} \rho(x, y)$ and the $\e$-inflation of the set by $C^\e = \{ x \in X : \rho(x, C) \le \e \}$.
Theorem~\ref{theorem:convergence} and Strassen's theorem \citep{strassen:1965} ensures uniformly for any closed Borel measurable set $C \subseteq \X$,
$
\mu\A^{(t)}(C)
\to \pi(C^\e).
$
Theorem~\ref{theorem:convergence} also ensures that for every bounded $\rho$-Lipschitz Borel measurable function $\phi : \X \to \R$,
\[
\int_{\X} \phi d\mu\A^{(t)}
\to \int_{\X} \phi d\pi.
\]

The following extends Theorem~\ref{theorem:convergence} to $L^p$-Wasserstein distances with unbounded metrics.

\begin{proposition}
\label{proposition:convergence_wasserstein}
Suppose an adaptive process $(\Gamma_t, X_t)_{t \ge 0}$ with initialization probability measure $\mu$ \eqref{eq:adaptive_process} satisfies weak containment \eqref{assumption:weak_containment} and weak diminishing adaptation \eqref{assumption:weak_da}.
Suppose further for some $x_0 \in \X$ and $p \in \Z_+$, 
$\int \rho(x, x_0)^p \pi(dx) < \infty$.
Then the following are equivalent:

\begin{enumerate}
\item 
\label{equiv:convergence}
Convergence in the $L^p$-Wasserstein distance holds:
$
\lim_{t \to \infty} 
\W_{\rho, p}\left(
\mu\A^{(t)}, \pi
\right)
= 0.
$

\item The sequence $( \rho(X_t, x_0)^p )_{t \ge 0}$ is uniformly integrable:
\[
\lim_{R \to \infty} \sup_{t \ge 0} \int_{ \rho(x, x_0) > R } \rho(x, x_0)^p \mu\A^{(t)}(dx)
= 0.
\]

\noindent If $(\X, \rho)$ is a complete separable metric space, then the following are also equivalent to \eqref{equiv:convergence}:

\item
$
\lim_{t \to \infty} \int_{\X} \rho(x, x_0)^p \mu\A^{(t)}(dx)
=  \int_{\X} \rho(x, x_0)^p \pi(dx)
$

\item 
$
\limsup_{t \to \infty} \int_{\X} \rho(x, x_0)^p \mu\A^{(t)}(dx)
\le  \int_{\X} \rho(x, x_0)^p \pi(dx)
$
\end{enumerate}
\end{proposition}

\begin{proof}
By Theorem~\ref{theorem:convergence} and Markov's inequality, for any $\e > 0$ 
$
\inf_{\xi \in \C(\mu\A^{(t)}, \pi)} \xi( \{ \rho(u, v) > \e \} )
\to 0.
$
Let $\xi^{(t)}$ be the attained optimal coupling for each $t$ \citep[Theorem 4.1]{Villani2008}.

Assume (1) holds.
For any $\e \in (0 , 1)$,
$\lim_{t \to \infty} \xi^{(t)}( | \rho(x, x_0) - \rho(y, x_0) | \ge \e) = 0$. 
By Young's inequality, for any $\e \in (0, 1)$, there is a constant $C_{\e, p}$ depending on $p, \e$ such that
\[
\rho(x, x_0)^p
\le (1 + \e) \rho(y, x_0)^p + C_{\e, p} \rho(x, y)^p
\]
Integrating with the coupling implies that
\[
\limsup_{t \to \infty} \int_\X \rho(x, x_0)^p \mu \A^{(t)}(dx)
\le \int_\X \rho(y, x_0)^p \pi(dy).
\]
By \citep[Theorem 5.11]{kallenberg:2021}, (2) holds.

Now assume (2) holds. By convexity,
\begin{align}
&\lim\sup_{t} \int_{\rho(x, y) > R} \rho(x, y)^p \xi^{(t)}(dx, dy)
\\
&\le 2^{p - 1} \lim\sup_{t} \int_{\rho(x, y) > R} \rho(x, x_0)^p \mu \A^{(t)}(dx) + 2^{p - 1} \lim\sup_{t} \int_{\rho(x, y) > R} \rho(y, x_0)^p \xi^{(t)}(dx, dy)
\label{eq:ui_ub}
\end{align}
By the characterization of uniform integrability \citep[Theorem 5.11]{kallenberg:2021} and dominated convergence, \eqref{eq:ui_ub} tends to $0$ as $R \to 0$.
This implies
\[
\lim_{t \to \infty} \W_{\rho, p}^p\left[
\mu\A^{(t)}, \pi
\right]
= \lim_{t \to \infty} \int_{\X^2} \rho(x_t, y)^p \xi^{(t)}(dx', dy')
= 0.
\]
The remaining equivalences follow from \citep[Lemma 5.11]{kallenberg:2021}.
\end{proof}

Proposition~\ref{proposition:convergence_wasserstein} has many interesting applications to extend weak convergence of adaptive MCMC and also extending convergence in total variation of adaptive MCMC.
For example, if $\X = \R^d$, then weak convergence from Theorem~\ref{theorem:convergence} can be used to extend the convergence to the $L^2$ Wasserstein distance $\W_{\norm{\cdot}, 2}$.
Another possibility is to extend traditional convergence of adaptive MCMC to stronger convergence \citep[Theorem 13]{roberts:rosenthal:2007} in the case when (strong) containment and (strong) diminishing adaptation hold.
The following corollary extends convergence in total variation to a stronger Wasserstein distance under similar conditions \citep[Theorem 18]{roberts:rosenthal:2007}. 

%For example, given a lower semi-continuous $V : \X \to [0, \infty)$, $\bar{\rho}(x, y) = \sqrt{ 1 + V(x) + V(y) ]I_{x \not= y} }$, defines a metric and $\W_{\bar{\rho}}$ defines a weighted total variation distance \citep{hairer:2011}. 
%Proposition~\ref{proposition:convergence_wasserstein} may be applied to extend convergence in total variation to even stronger weighted total variation distances. 

\begin{corollary}
\label{corollary:weighted_tv_convergence}
Suppose an adaptive process $(\Gamma_t, X_t)_{t \ge 0}$ with initialization probability measure $\mu$ \hl{as in} \eqref{eq:adaptive_process} satisfies (strong) containment \eqref{assumption:strong_containment} and (strong) diminishing adaptation \eqref{assumption:strong_da}.
Suppose there is a lower semicontinuous function $V : \X \to [0, \infty)$ and constants $\lambda \in (0, 1)$ and $L \in (0, \infty)$ such that for all $\gamma, x \in \Y \times \X$,
$
(\P_\gamma V) (x) \le \lambda V(x) + L.
$
If $\int_{\X} V d\mu < \infty$, then
\begin{align*}
&\lim_{t \to \infty} \W_{\bar{\rho}}\left( \mu \A^{(t)}, \pi \right)
= 0
\end{align*}
where $\bar{\rho}(x, y) = \left[ ( 1 + V(x) + V(y) ) \right]^{1/2}$ if $x \not= y$ and $0$ otherwise.
\end{corollary}
\begin{proof}
The drift condition and assumption on $\mu$ imply $ ( \sqrt{ V(X_t) } )_{t \ge 0}$ is uniformly integrable and Proposition~\ref{proposition:convergence_wasserstein} implies the conclusion.
\end{proof}

\begin{remark}
An alternative way to extend weak convergence to a stronger convergence in total variation convergence is through addition of an independent random variable \citep{bogachev:2018}.
Consider $\X = \R^d$ for some $d \in \Z_+$, and an adaptive process $(\Gamma_t, X_t)_{t \ge 0}$ with initialization probability measure $\mu$ \eqref{eq:adaptive_process} satisfies weak containment \eqref{assumption:weak_containment} and weak diminishing adaptation \eqref{assumption:weak_da} both with metric $\rho(\cdot, \cdot) = \norm{\cdot - \cdot}$.
Let $h \in (0, 1)$ and let $\sigma_h$ be a Gaussian distribution on $\R^d$ $N(0, h I)$.
Then
\[
\lim_{t \to \infty} \W_{\text{TV}}\left( \mu \A^{(t)} * \sigma_h, \pi * \sigma_h \right)
= 0
\]
where $*$ denotes the convolution.
\end{remark}

The following is a useful coupling technique to show weak containment \eqref{assumption:weak_containment}.

\begin{lemma}
\label{lemma:contract_containment}
Let $(\Gamma_t, X_t)_{t \ge 0}$ be an adaptive process with initialization probability measure $\mu$ \hl{as in} \eqref{eq:adaptive_process}.
Suppose $\pi \P_\gamma = \pi$ for every $\gamma \in \Y$. Assume for some $x_0 \in \X$ and some $p \in \Z_+$, $L = \int_\X \rho(x, x_0)^p \pi(dx) < \infty$ and for every $\gamma, x \in \Y \times \X$, $\int_\X \rho(y, x_0)^p \P_{\gamma}(x, dy) < \infty$.
Suppose there is an $\alpha \in (0, 1)$ such that for all $x, y \in \X$ and $\gamma \in \Y$, 
\[
\W_{\rho, p}(\P_\gamma(x, \cdot), \P_\gamma(y, \cdot))
\le (1 - \alpha) \rho(x, y).
\]
Then for every $t \in \Z_+$ and $x \in \X$,
\[
\W_{\rho, p}(\P^t_{\gamma}(x, \cdot), \pi) \le (1 - \alpha)^t [ \rho(x, x_0) + L ].
\]
Further, $\sup_{t \ge 0} \E\left[ \rho(X_t, x_0)^p \right] < \infty$ and $( \rho(X_t, x_0) )_{t \ge 0}$ is bounded in probability.
\end{lemma}

\begin{proof}
For each $t \in Z_+$ and each $x, y \in \X$, we have
\[
\W_{\rho, p}(\P^t_\gamma(x, \cdot), \P^t_\gamma(y, \cdot))^p
\le (1 - \alpha)^{t p} \rho(x, y)^p.
\]
The optimal coupling is attained \citep[Theorem 4.1]{Villani2008} at some conditional coupling $\xi_{x, y}$ and is Borel measurable \citep[Corollary 5.22]{Villani2008}.
Since $\int_\X \xi_{x, y}(\cdot, \cdot) \pi(dy) \in \C[\P_\gamma^t, \pi]$, then
\begin{align*}
\W_{\rho, p}(\P^t_\gamma(x, \cdot), \pi)^p
&\le \int_\X \int_{\X^2} \rho(x', y')^p \xi_{x, y}(dx', dy') \pi(dy)
\\
&\le (1 - \alpha)^{t p} \int_\X \rho(x, y)^p \pi(dy).
\end{align*}
By Young's inequality, for any $\e > 0$, there is a constant $C_\e > 0$ such that for any $a, b \ge 0$, $(a + b)^p \le (1 + \e) a^p + C_\e b^p$.
For any $x_0 \in \X$, we can choose $\e$ sufficiently small so that $(1 + \e)(1 - \alpha)^{p} < 1$ and a constant $C_{\e, p}$ such that 
\begin{align*}
\int_\X \rho(x', x_0)^p \P_\gamma(x, dx')
\le \left[ (1 - \alpha) \rho(x, x_0)
+ 2 L^{1/p} \right]^p
\\
\le (1 + \e) (1 - \alpha)^{p} \rho(x, x_0)^p
+ C_{\e, p} L.
\end{align*}
By \citep[Lemma 15]{roberts:rosenthal:2007}, this simultaneous geometric drift condition implies that $( \rho(X_t, x_0) )_{t \ge 0}$ is bounded in probability.
\end{proof}

\section{A weak law of large numbers}
\label{section:lln}

The point of this section is to develop convergence in probability of the empirical average of the adaptive MCMC process or weak law of large numbers.
The convergence theory developed so far in Wasserstein distances provides estimation accuracy for the marginal distribution of $X_t$ but this generally has a large variability.
Estimation from the entire adaptive process $X_s \sim \mu \A^{(s)}$ for $s \le t$ requires theory for the empirical average.
It is then of interest for reliable estimation to develop conditions so that for bounded $\rho$-Lipschitz functions
\[
\frac{1}{t} \sum_{s = 1}^t \phi(X_s)
\to \int_{\X} \phi d\pi
\]
in probability.

The law of large numbers for non-adapted Markov chains is well-studied under convergence in total variation. 
On the other hand, convergence in Wasserstein distances and its connection to the law of large numbers is less understood \citep[Theorem 1.2]{sandric:2017}.
The first result is general and relies on the convergence of the adaptive process but may even apply the law of large numbers to unbounded functions if the conditions are satisfied.

\begin{theorem}
\label{theorem:wlln}
Let $(\Gamma_t, X_t)_{t \ge 0}$ be an adaptive process with initialization probability measure $\mu$ so that $X_t \sim \mu \A^{(t)}$ \eqref{eq:adaptive_process}.
Let $d(\cdot, \cdot)$ be a lower semicontinuous metric and suppose for some $x_0 \in \X$ and for each $t \in \Z_+$,
$
\int d(x, x_0)^2 \mu \A^{(t)}(dx) < \infty
$
and
$
\int d(x, x_0)^2 \pi(dx) < \infty.
$
If
\[
\lim_{t \to \infty} \W_{d, 2}\left( \mu \A^{(t)}, \pi \right)
= 0,
\]
then for every Borel measurable $\phi : \X \to \R$ with $\norm{\phi}_{\text{Lip}(d)} < \infty$,
\begin{align}
\lim_{t \to \infty} \E\left[ 
\left(
\frac{1}{t} \sum_{s = 1}^t \phi(X_s)
- \int_\X \phi d\pi
\right)^2
\right]
= 0.
\label{eq:wlln_result}
\end{align}
In particular, the weak law of large numbers holds, that is, $\frac{1}{t} \sum_{s = 1}^t \phi(X_s) \to \int_{\X} \phi d\pi$ in probability.
\end{theorem}

\begin{proof}
We can assume $\norm{\phi}_{\text{Lip}(d)} \le 1$ since we may normalize $\phi$.
We may also assume $\int \phi d\pi = 0$ since $\psi = \phi - \int \phi d\pi$ is also $d$-Lipschitz.
We can assume there is an $x_0 \in \X$ such that $\phi(x_0) = 0$.
Let $\Gamma$ be a coupling of $X_t \sim \mu \A^{(t)}$ and $Y \sim \pi$.
By disintegration \citep[Theorem 3.4]{kallenberg:2021}, there is a Borel measurable conditional probability measure $\Gamma_{h_s}(dx_t, dy)$ with $h_s = (\gamma_1, x_1, \ldots, \gamma_s, x_s)$ such that
\begin{align*}
\Gamma(dx_t, dy)
= \int_{\X^s} \Gamma_{h_s}(dx_t, dy) \mu\A^{(1, \ldots, s)}(dh_s).
\end{align*}
With the history $H_s = (\Gamma_k, X_k)_{k = 1}^s$ and since $\phi$ is $d$-Lipschitz, for $t \ge s$
\[
\left| \E\left[ \phi(X_{t}) | \H_s \right] - \int_{\X} \phi d\pi \right|
\le \int_{\X^2} d(x_{t}, y) \Gamma_{H_s}(dx_{t}, dy).
\]
For $T \in \Z_+$, we have the upper bound
\begin{align*}
\E\left[ \left( \sum_{t=1}^T \phi(X_t)
\right)^2 \right]
&= \sum_{t=1}^T \sum_{s=1}^T 
\E \left[ \phi(X_t) \phi(X_s) \right]
\\
&= \sum_{t=1}^T \E \left[ \phi(X_t)^2 \right]
+ 2 \sum_{t=2}^T \sum_{s=1}^{t-1} 
\E\left[ \E\left[ \phi(X_t) | \H_s \right] \phi(X_s) \right]
\\
&\le T \sup_{t \ge 0} \int_\X d(x, x_0)^2 \mu\A^{(t)}(dx)
\\
&+ 2 \sum_{t=2}^T \sum_{s=1}^{t-1} 
\E\left[ \int_{\X^2} d(x_{t}, y) \Gamma_{H_s}(dx_{t}, dy) d(X_s, x_0) \right].
\end{align*}
Using Cauchy-Schwarz and Jensen's inequality,
\begin{align*}
&\E\left[ \int_{\X^2} d(x_{t}, y) \Gamma_{H_s}(dx_{t}, dy) d(X_s, x_0) \right]
\\
&\le \sqrt{ \E\left[ \left[ \int_{\X^2} d(x_{t}, y) \Gamma_{H_s}(dx_{t}, dy) \right]^2 \right] } 
\sqrt{ \E\left[ d(X_s, x_0)^2 \right] }
\\
&\le \sqrt{ \int_{\X \times \X} d(x_{t}, y)^2 \Gamma(dx_{t}, dy) }
\sqrt{ \sup_{t \ge 0} \int_\X d(x, x_0)^2 \mu\A^{(t)}(dx) }.
\end{align*}

By assumption, we can choose a $T_\e$ depending on $\e$ such that for all $t \ge T_\e$,
\[
\W_{d, 2}\left( \mu \A^{(t)}, \pi \right)
\le \e.
\]
By assumption, $\max_{0 \le t \le T_\e} \E \left[ d(X_t, x_0)^2 \right] < \infty$ and it follows by the triangle inequality \citep[Lemma 7.6]{Villani2003} that there is an $R \in (0, \infty)$ such that
\begin{align*}
\sup_{t \ge 0} \W_{d, 2}\left( \mu \A^{(t)}, \pi \right)
\le \sqrt{ \sup_{t \ge 0} \E\left( d(X_t, x_0)^2 \right) }
+ \sqrt{ \int d(x, x_0)^2 \pi(dx) }
\le R.
\end{align*} 
Since the coupling $\Gamma$ is arbitrary, we have the upper bound for every $T \ge T_\e + 1$,
\begin{align*}
\E\left[ \left( \frac{1}{T} \sum_{t=1}^T \phi(X_t)
\right)^2 \right]
&\le \frac{R^2}{T} 
+ \frac{2 R}{T^2} \sum_{t=2}^T (t-1) 
\W_{d, 2}\left( \mu \A^{(t)}, \pi \right)
\\
&\le \frac{R^2}{T}
+ \frac{2 R^2 }{T^2} \sum_{t=2}^{T_\e} (t-1)
+ 
\frac{2 R}{T^2} \sum_{t= T_{\e} + 1}^T (t-1) 
\e
\\
&\le \frac{R^2}{T}
+ \frac{ R^2 T_{\e} (T_{\e} - 1) }{T^2}
+ 
\frac{R \e (T - T_\e) (T + T_\e - 1)}{T^2}.
\end{align*}
The conclusion follows since we can choose $T$ sufficiently large and $\e$ sufficiently small.
\end{proof}

Theorem~\ref{theorem:wlln} also provides general conditions for weakly converging Markov chains \citep[Theorem 1.2]{sandric:2017}.
We may now show that the conditions of Theorem~\ref{theorem:convergence} are sufficient for a weak law of large numbers for bounded $\rho-$Lipschitz functions.

\begin{corollary}
\label{corollary:wlln}
Suppose an adaptive process $(\Gamma_t, X_t)_{t \ge 0}$ with initialization probability measure $\mu$ \eqref{eq:adaptive_process} satisfies weak containment \eqref{assumption:weak_containment} and weak diminishing adaptation \eqref{assumption:weak_da}.
Then for every bounded Borel measurable $\phi : \X \to \R$ with $\norm{\phi}_{\text{Lip}(\rho)} < \infty$ and any $p \in \Z_+$
\begin{align}
\lim_{t \to \infty} \E\left[ 
\left|
\frac{1}{t} \sum_{s = 1}^t \phi(X_s)
- \int_\X \phi d\pi
\right|^p
\right]
= 0.
\label{eq:lln_bl}
\end{align}
If in addition, for some $x_0 \in \X$, $( \rho(X_t, x_0)^2 )_{t \ge 0}$ is uniformly integrable and $\int \rho(x, x_0) \pi(dx) < \infty$, then \eqref{eq:lln_bl} holds with $p = 2$ and for all $\norm{\phi}_{\text{Lip}(\rho)} < \infty$.
\end{corollary}

\begin{proof}
For Borel measurable $\phi : \X \to \R$ with $\norm{\phi}_{\text{Lip}(\rho)} < \infty$, apply Theorem~\ref{theorem:convergence} and Theorem~\ref{theorem:wlln} and it follows \eqref{eq:wlln_result} holds. Since $\phi$ is bounded and \eqref{eq:wlln_result} holds, $L^p$ convergence follows for $p \in \Z_+$.
To remove the bounded assumption on $\phi$, since it is assumed $( \rho(X_t, x_0)^2 )_{t \ge 0}$ is uniformly integrable, $\E(\rho(X_t, x_0)^2) < \infty$ for each $t \in \Z_+$ and $\int \rho(x, x_0) \pi(dx) < \infty$ then we can apply Proposition~\ref{proposition:convergence_wasserstein}.
\end{proof}

%One approach to verifying that $( \rho(X_t, x_0)^2 )_{t \ge 0}$ is uniformly integrable is through simultaneous drift conditions as defined in \citep{roberts:rosenthal:2007}.
%In particular, \citep[Theorem 18]{roberts:rosenthal:2007} then guarantees a weak law of large numbers for functions that are dominated by the squre root of the drift function and possibly unbounded.

%Stronger conditions allow for unbounded Lipschitz functions which can be controlled with simultaneous drift conditions.
%The additional integrability conditions in Theorem~\ref{theorem:wlln} can be controlled using drift conditions. 
%In particular, Corollary~\eqref{corollary:wlln} can be used in conjunction with Corollary~\ref{corollary:weighted_tv_convergence} existing results when Wasserstein distance is related to a weighted total variation \citep[Theorem 23]{roberts:rosenthal:2007}.

\section{Examples and applications}
\label{section:examples_applications}

Let us now revisit constructing adaptive processes for the running examples of Markov chains \ref{example:discrete_ar_process} and \ref{example:infinite_ar_process} where (strong) containment \eqref{assumption:strong_containment} fails to hold.

\subsection{Example: discrete adaptive auto-regressive process}
\label{example:discrete_ar_process_convergence}

Consider an adaptive process $(\Gamma_t, X_t)_{t \ge 0}$ using the Markov kernels $(\P_\gamma)_\gamma$ for the discrete auto-regressive process defined in \hl{Example~\ref{example:discrete_ar_process}} which fails to satisfy (strong) containment \eqref{assumption:strong_containment}.
Assume the adaptation satisfies $|\Gamma_{t + 1} - \Gamma_{t}| \to 0$ in probability as $t \to \infty$.
For any $\gamma \in \Z_+$ with $\gamma \ge 2$ and $x \in [0, 1)$, we showed previously \hl{in Example~\ref{example:discrete_ar_process}} that for any $t \in \Z_+$
\[
\W_{ |\cdot|}\left( \P^t_\gamma(x, \cdot), \text{Unif}(0, 1) \right) \le 2^{-t}
\]
where $\text{Unif}([0, 1)$ is Lebesgue measure on $[0, 1)$ and so weak containment holds \eqref{assumption:weak_containment}.
For every $x, y \in [0, 1)$ define $X^{\Gamma_{t}}_1 = x/\Gamma_t + \xi^{\Gamma_t}_1$ and $Y^{\Gamma_{t}}_1 = y/\Gamma_t + \xi^{\Gamma_t}_1$ with common discrete random variable $\xi^{\Gamma_t}_1$ defined previously in \hl{Example~\ref{example:discrete_ar_process}}.
The random variables $(X^{\Gamma_{t + 1}}_1, Y^{\Gamma_{t}}_1)$ define a coupling and since for sufficiently large $t$, $\Gamma_{t + 1} = \Gamma_t$ with high probability, then 
\begin{align}
\sup_{|x - y| \le \delta} \W_{ |\cdot|}\left( \P_{\Gamma_{t + 1}}(x, \cdot), \P_{\Gamma_{t}}(y, \cdot) \right)
\le \sup_{|x - y| \le \delta} \E\left[ \left| X^{\Gamma_{t + 1}}_1 - Y^{\Gamma_{t}}_1 \right| | X_0 = x, Y_0 = y \right]
\le \frac{\delta}{2}
\label{eq:discrete_ar_ub}
\end{align}
holds with high probability.
Then \eqref{eq:discrete_ar_ub} tends to $0$ as $\delta \to 0$ and we conclude weak diminishing adaptation \eqref{assumption:weak_da} holds. 
By Theorem~\ref{theorem:convergence}, for every $\gamma \ge 2$ and $x \in [0, 1)$,
\[
\lim_{t\to\infty} \W_{|\cdot|}[ \A^{(t)}( (\gamma, x) , \cdot), \text{Unif}(0, 1) ]
= 0
\]
and this discrete auto-regressive adaptive process converges weakly.
Corollary~\ref{corollary:wlln} then implies a weak law of large numbers for all bounded Lipschitz continuous functions.

\subsection{Example: infinite-dimensional adaptive auto-regressive process}
\label{example:infinite_ar_process_convergence}

Consider an adaptive process $(\Gamma_t, X_t)_{t \ge 0}$ using Markov kernels $(\P_\gamma)_\gamma$ for the infinite-dimensional auto-regressive process \eqref{example:infinite_ar_process} which cannot satisfy (strong) containment \eqref{assumption:strong_containment}.
Assume the adaptation is restricted to a bounded set, that is, for some $R \in (0, \infty)$, if $\norm{X_t} > R$ then $\Gamma_{t + 1}= \Gamma_{t}$ and $|\Gamma_t - \Gamma_{t + 1}| \to 0$ in probability as $t \to \infty$.

We showed previously \eqref{example:infinite_ar_process} that for any $\gamma \in (0, \gamma^*)$ and any $x, y \in H$,
\begin{align}
\W_{\norm{\cdot}, 2}( \P_{\gamma}(x, \cdot), \P_{\gamma}(y, \cdot) ) 
\le \gamma^* \norm{x - y}
\label{eq:infinit_ar_ub}
\end{align}
and combined with Lemma~\ref{lemma:contract_containment} implies weak containment \eqref{assumption:weak_containment} holds.
For $x, y \in H$, let $Y_t = \Gamma_{t + 1} x + \sqrt{1 - {\Gamma_{t + 1}}^2} \xi_{t}$ and $Y'_t = \Gamma_{t}  y + \sqrt{1 - {\Gamma_{t}}^2} \xi_{t}$ with common independent random variable $\xi_t \sim \N(0, C)$.
We have the upper bound for any $t \in \Z_+$
\begin{align*}
\W_{\norm{\cdot}, 2}( \P_{\Gamma_{t + 1}}(x, \cdot), \P_{\Gamma_{t}}(y, \cdot) )
&\le \left[ \E\left( \norm{Y_t - \Gamma_{t}x + \Gamma_{t}x - Y'_t}^2 \right) \right]^{1/2}
\\
&\le | \Gamma_{t + 1} - \Gamma_{t} | \norm{ x }  
+ {\gamma^*} \norm{ x - y }
\\
&\hspace{1.2em}+ \left| \sqrt{1 - \Gamma_{t + 1}^2} - \sqrt{1 - \Gamma_{t}^2} \right| \sqrt{ \tr(C) }.
\end{align*}
If $\norm{x} > R$ and $\norm{x - y} \le \delta$, then
$
\W_{\norm{\cdot}, 2}( \P_{\Gamma_{t + 1}}(x, \cdot), \P_{\Gamma_{t}}(y, \cdot) )
\le \gamma^* \delta.
$
Otherwise if $\norm{x} \le R$ and $\norm{x - y} \le \delta$,
\[
\W_{\norm{\cdot}, 2}( \P_{\Gamma_{t + 1}}(x, \cdot), \P_{\Gamma_{t}}(y, \cdot) )
\le | \Gamma_{t + 1} - \Gamma_{t} | R 
+ \gamma^* \delta
+ \left| \sqrt{1 - \Gamma_{t + 1}^2} - \sqrt{1 - \Gamma_{t}^2} \right| \sqrt{\tr(C)}.
\]
In either case, weak diminishing adaptation \eqref{assumption:weak_da} holds.

For any $p \in \Z_+$, Young's inequality and Fernique's theorem \citep[Theorem 2.8.5]{bogachev:1998} imply we can choose $\e >0 $ such that $(1 + \e) {\gamma^*}^p < 1$ and a constant $C_{\e} > 0$ depending on $\e$ such that for every $\gamma, x \in \Y \times \X$
\begin{align*}
(\P_\gamma) ( \norm{ x }^p )
\le (1 + \e) {\gamma^*}^p \norm{x}^p + C_{\e}.
\end{align*}
This implies $( \norm{X_t}^p ) _{t \ge 0}$ is uniformly integrable.
From Theorem~\ref{theorem:convergence} and Proposition~\ref{proposition:convergence_wasserstein}, then for every $\gamma, x \in (0, \gamma^*] \times H$ and every $p \in \Z_+$,
\[
\lim_{t\to\infty} \W_{\norm{\cdot}, p}[ \A^{(t)}( (\gamma, x), \cdot), \N(0, \C) ]
= 0
\]
and Corollary~\ref{corollary:wlln} implies a weak law of large numbers for all Lipschitz continuous functions.

\subsection{Example: adaptive random-walk Metropolis-Hastings}
\label{section:amh}

We look at a discrete version of the adaptive random-walk Metropolis-Hastings algorithm \citep{haario:2001} which adapts the covariance of the proposal towards the covariance of the target probability measure.
This concrete example illustrates an issue in practical applications since current computers only produce floating-point approximations to real numbers. As a result, convergence theory in total variation corresponding to an adaptive Markov chain Monte Carlo simulation targeting a continuous target probability measure is infeasible and has other issues \hl{that} have been studied previously through perturbation theory \citep{roberts:rosenthal:schwartz:1998, breyer:etal:2001}.
This is not necessarily the case in alternative distances which metrize weak convergence.

Let $\pi$ be a target Borel probability measure on $\R^d$ with $d \in \Z_+$ and Lebesgue density $f$.
Let $D = (x_k)_{k = 1}^{\infty} \subset \R^d$ be a dense subset in $\R^d$, $\Sigma$ be a symmetric, positive-definite matrix, and $\mu \in \R^d$, and let $\N_{D}(\mu, \Sigma)$ denote the discrete Gaussian with probability mass function
\[
g_{\Sigma}(\mu, x) = \frac{\exp\left( -\frac{1}{2} ( x - \mu )^T \Sigma^{-1} ( x - \mu ) \right)}{\sum_{j = 1}^{\infty} \exp\left( - \frac{1}{2} ( x_j - \mu )^T \Sigma^{-1} ( x_j - \mu ) \right) }.
\]
When $\mu \in D$, the $g_{\Sigma}(\mu, x) = g_{\Sigma}(x, \mu)$ and it is symmetric.
Let $(M(\gamma))_{\gamma \in \Y}$ be a family of symmetric, positive-definite matrices on $\R^d$.
We define a discrete Markov chain $( X^{\gamma}_{t} )_{t \ge 0}$ using a discrete random-walk Metropolis-Hastings kernel $\P_{\Sigma}$ with discrete Gaussian proposal where the proposal $X'$ given the previous state $x$ is $X' \sim \N_{D}(x, M(\gamma))$.
The Markov kernel is defined for each $x_l, x_k \in D$ by
\[
\P_{\gamma}(x_l, x_k)
= \left[ 1 \wedge \frac{f(x_k)}{f(x_l)} \right] g_{M(\gamma)}(x_l, x_k)
+ \delta_{x_l}(\{ x_k \})\left( 1 - \sum_{j = 1}^{\infty} \left[ 1 \wedge \frac{f(x_j)}{f(x_l)} \right] g_{M(\gamma)}(x_l, x_j) \right).
\]

We will assume $f$ is continuous with compact support $K \subset \R^d$.
Further, we will assume the set of $\Sigma$  is compact and so the eigenvalues of $\Sigma$ are uniformly bounded so that there are constants $\lambda_*, \lambda^* \in (0, \infty)$ such that $\lambda_* \le \lambda_{i}(\Sigma) \le \lambda^*$ for all $i = 1, \ldots, d$.
It follows by a a minorization argument over $K$ there is an $\alpha \in (0, 1)$ such that for any $\gamma \in \Y$, any $x_i, x_j \in D$, and any bounded Lipschitz continuous function $\phi : \R^d \to \R$,
\[
\left| \P_\gamma^t \phi(x_i) - \P_\gamma^t \phi(x_j) \right|
\le (1 - \alpha)^t.
\]
By the density of $D$ and continuity of $\P \phi(\cdot)$, we have that 
\[
\left| \P_\gamma^t \phi (x_i) - \int_K \phi d\pi \right|
\le (1 - \alpha)^t.
\]
Weak containment \eqref{assumption:weak_containment} holds since it follows by Kantorovich-Rubinstein duality \citep[Theorem 1.14]{Villani2003}
\[
\W_{\norm{\cdot} \wedge 1}\left( \P_\gamma^t(x_i, \cdot), \pi \right)
\le (1 - \alpha)^t.
\]

Now let $\gamma, x \in \Y \times D$ and let $(\Gamma_t, X_t)_{t \ge 0}$ where $X_t \sim \A^{(t)}( (\gamma, x), \cdot)$ be the adaptive process using these Metropolis-Hastings kernels.
Under any valid weak diminishing adaptation strategy satisfying \eqref{assumption:weak_da}, we have
\[
\lim_{t \to \infty} \W_{\norm{\cdot} \wedge 1}\left( \A^{(t)}( (\gamma, x), \cdot), \pi \right)
= 0.
\]
Corollary~\ref{corollary:wlln} then implies a weak law of large numbers for bounded Lipschitz continuous functions.
On the other hand, it can be shown that
\[
\W_{\text{TV}}(\A^{(t)}\left( (\gamma, x), \cdot), \pi \right) = 1
\]
and fails to converge in total variation under any adaptation plan.

\subsection{Example: adaptive unadjusted Langevin process}
\label{example:adaptive_ula}

In certain cases, it has been observed \citep[page 21]{Villani2008} that Wasserstein distances can be simpler to prove convergence results than total variation and the following example provides a concrete illustration.
Consider the Euclidean space $\R^d$ where $d \in \Z_+$ with Euclidean norm $\norm{\cdot}$. 
Let the potential $V : \R^d \to \R$ have gradient $\nabla V(\cdot)$ with constants $\alpha, \beta > 0$ such that for every $x, y \in \R^d$,
\begin{align}
&\norm{ \nabla V(y) - \nabla V(x) }
\le \beta \norm{y - x}
\label{eq:lip_grad}
\\
&V(y)
\ge V(x) + \ip{\nabla V(x), y - x} + \frac{\alpha}{2} \norm{y - x}^2.
\label{eq:strong_convex}
\end{align}
Let $(\Y, d)$ be a complete separable metric space and for each $\gamma \in \Y$, let $M(\gamma)$ define a symmetric positive definite matrix.
Let $h \in (0, 1)$ be a fixed discretization size and consider the unadjusted Langevin process
\[
X^{\gamma, h}_{t + 1}
= X^{\gamma, h}_t - h M(\gamma) \nabla V(M(\gamma) X^{\gamma, h}_t) + \sqrt{2 h} Z_{t + 1}
\]
where $(Z_t)_{t \ge 0}$ are i.i.d $\N(0, I_d)$.
We can define a family of Markov kernels $( \P_{\gamma, h} )_{\gamma, h}$ prescribing the conditional distributions $X^{\gamma, h}_{t + 1} | X^{\gamma, h}_{t} = x \sim \P_{\gamma, h}(x, \cdot)$.
For an adaptive strategy $(\Gamma_t, h_t)_{t \ge 0}$, we define the adaptive process $X_t := M(\Gamma_t) X^{\Gamma_t, h_t}_t$ for $t \ge 0$.
For example, $M(\Gamma_t)$ can estimate the inverse covariance matrix using the entire history of the process as in adaptive Metropolis-Hastings \citep{haario:2001} and adaptive Piecewise-deterministic Markov processes \citep{bertazzi:2022}.
We make the following assumption on the \hl{adaptation of} the matrix $M(\gamma)$.
\begin{assumption}
\label{assumption:adaptation_plan_ula}
Assume $M(\cdot)$ is continuous and for each $\gamma \in \Y$ and $x \in \R^d$, there is a constant $\lambda_* \in (0, \infty)$ such that for every $v \in \R^d\setminus \{0\}$ with $\norm{v} = 1$,
$
\lambda_* \le \ip{ v M(\gamma), v } \le 1.
$
Assume $d(\Gamma_{t + 1}, \Gamma_t) \to 0$ in probability and $\Gamma_{t + 1} = \Gamma_t$ if $\norm{X_{t}} > R$ for some $R > 0$.
\end{assumption}

\noindent We have the following convergence result.

\begin{proposition}
Assume the adaptation plan satisfies assumption~\ref{assumption:adaptation_plan_ula} and for some $h_* \in (0, 1)$, let $H = [h_*, 1/(\alpha + \beta) ]$ and assume $(h_t)_{t \ge 0}$ is a deterministic sequence with $h_t \in H$. Further, assume there is a limit $h^* \in H$ such that $\lim_{t \to \infty}|h_t - h^*| = 0$. Then the adaptive unadjusted Langevin process converges in the $L^1$-Wasserstein distance for every $\gamma, h, x \in \Y \times H \times \R^d$ to some probability measure $\pi_{h^*}$, that is,
\[
\lim_{t\to\infty} \W_{\norm{\cdot}}[ \A^{(t)}( (\gamma, h, x), \cdot), \pi_{h^*} ]
= 0
\]
and a weak law of large numbers holds for all bounded Lipschitz continuous functions.
\end{proposition}

\begin{proof}
It can be shown that $M(\gamma) \nabla V(M(\gamma) \cdot )$ is $\beta$ Lipschitz and $\alpha$ strongly convex.
For $x, y \in \R^d$ and $\gamma, \gamma' \in \Y$, let 
\begin{align*}
&X^{\gamma}_{1}
= x - h M(\gamma) \nabla V(M(\gamma) x) + \sqrt{2 h} Z_{1}
&Y^{\gamma'}_{1}
= y - h M(\gamma') \nabla V(M(\gamma') y) + \sqrt{2 h} Z_{1}
\end{align*}
with shared Gaussian random variable $Z_1 \sim N(0, I)$.
By \citep[Theorem 2.1.12]{nesterov:2018}, then
\begin{align*}
\E\left[ \norm{X^{\gamma, h}_1 - Y^{\gamma, h}_1}^2 \right]
&\le \left( 1 - \frac{2 h \alpha \beta}{\alpha + \beta} \right) \norm{x - y}^2
+  h \left( h - \frac{1}{\alpha + \beta} \right) \norm{\nabla V(x) - \nabla V(y)}^2
\\
&\le \left( 1 - \frac{2 h_* \alpha \beta}{\alpha + \beta} \right) \norm{x - y}^2.
\end{align*}
For each adapted discretization size $h, h'$, we have
\begin{align*}
&\left( \E\norm{M(\gamma) X^{\gamma, h'}_1 - M(\gamma)  X^{\gamma, h}_1 }^2 \right)^{1/2}
\le |h' - h| \beta \norm{x}.
\end{align*}
These imply along with the assumed adaptation strategy that there exists an invariant measure $\pi_{h^*}$ with finite second moment.
By Lemma~\ref{lemma:contract_containment}, weak containment \eqref{assumption:weak_containment} holds. 

For each adapted discretization size $h, h'$ and each adaptation parameter $\gamma, \gamma' \in \Y$, we have
\begin{align*}
&\left( \E\norm{M(\gamma') X^{\gamma', h'}_1 - M(\gamma)  Y^{\gamma, h}_1 }^2 \right)^{1/2}
\\
&\le 
\left( \E\norm{M(\gamma') X^{\gamma', h'}_1 - M(\gamma')  X^{\gamma', h}_1 }^2 \right)^{1/2}
+ \left( \E\norm{M(\gamma') X^{\gamma', h}_1 - M(\gamma)  X^{\gamma, h}_1 }^2  \right)^{1/2}
\\
&\hspace{1.5em}+ \left( \E\norm{X^{\gamma, h}_1 -  Y^{\gamma, h}_1 }^2  \right)^{1/2}
\\
&\le \beta |h' - h| \norm{x}
+ \norm{M(\gamma') \nabla V( M(\gamma') x ) - M(\gamma)  \nabla V( M(\gamma) x ) }
\\
&\hspace{1.5em}+ \norm{\nabla V( M(\gamma) x ) - \nabla V( M(\gamma) y ) }
\\
&\le \beta |h' - h| \norm{x}
+ 2 \beta \norm{M(\gamma') - M(\gamma) } \norm{ x }
+ \left( 1 - \frac{2 h_* \alpha \beta}{\alpha + \beta} \right) \norm{ x - y }.
\end{align*}
This upper bound implies weak diminishing adaptation \eqref{assumption:weak_da} under this adaptation strategy.
Therefore, this adaptive process converges weakly by Theorem~\ref{theorem:convergence}.
Lemma~\ref{lemma:contract_containment} implies $(\norm{X_t})_{n \ge 0}$ is uniformly integrable and then by Proposition~\ref{proposition:convergence_wasserstein}, the Wasserstein convergence follows.
Corollary~\ref{corollary:wlln} implies a weak law of large numbers for all bounded Lipschitz functions.
\end{proof}

\subsection{Example: adaptive diffusion process}
\label{section:example_adifussion}

\hl{Let the potential $V : \R^d \to \R$ satisfy \eqref{eq:lip_grad}.}
%Let $(\Y, d)$ be a complete separable metric space and for each $\gamma \in \Y$, let $M(\gamma)$ define a symmetric positive definite matrix. 
Let $( M(\gamma) )_{\gamma \in \Y}$ be defined as in Example~\ref{example:adaptive_ula} and consider adapting a stochastic differential equation with $M(\gamma)$ defined for $t \in [0, 1]$ by
\[
dX^{\gamma}_t = - M(\gamma) \nabla V\left( M(\gamma) X^{\gamma}_t \right) dt + \sqrt{2} dW_t
\]
where $(W_t)_{t \ge 0}$ is standard Brownian motion in $\R^d$.
Then for any $\gamma \in \Y$ and $x \in \R^d$, there exists a strong solution $(X^{\gamma}_t)_{t \ge 0}$ that is a Markov process with kernel $\tilde\P^t_{\gamma}$.
Using the solution at $t = 1$, we can define a new Markov chain $X^{\gamma}_{n} | X_0 = x$ with Markov transition kernel $\tilde{P}^1_{\gamma}$.
We can define an adaptive process $(\Gamma_n, X_n)_{n \ge 0}$ with $X_n = M(\Gamma_n) X_n^{\Gamma_n}$ so that $X_n | \Gamma_n = \gamma, X_{n - 1} = x$ has Markov transition $\P_{\gamma}(x, \cdot)$ with invariant measure $\pi(dx) = Z^{-1} \exp(-V(x)) dx$ where $Z = \int \exp(-V(x)) dx$.
This type of adaptive scheme using matrices has been successful in Piecewise-deterministic Markov processes \citep{bertazzi:2022}.

\begin{proposition}
Assume the adaptation plan satisfies Assumption~\ref{assumption:adaptation_plan_ula} and \hl{assume the} \hl{potential $V : \R^d \to \R$ satisfies \eqref{eq:lip_grad} and \eqref{eq:strong_convex}}.
Then for every $\gamma, x \in \Y \times \R^d$, the adaptive diffusion process converges in the $L^1$-Wasserstein distance
\[
\lim_{t\to\infty} \W_{\norm{\cdot}}[ \A^{(t)}( (\gamma, x), \cdot), \pi ]
= 0
\]
and a weak law of large numbers for all bounded Lipschitz continuous functions.
\end{proposition}

\begin{proof}
For $\gamma, \gamma' \in \Y$ and $x, y \in \R^d$, let $X^{\gamma}_t | X_0 = x$ and $Y^{\gamma'}_t | Y_0 = y$ and $Y_0 = y$ share the same Brownian motion so that these random variables define a coupling.
By strong convexity, for every $x, y \in \R^d$ $\ip{ \nabla V(x) - \nabla V(y), y - x } \ge \alpha \norm{x - y}^2$.
Define $f_\gamma(t) =  \norm{ X^{\gamma}_t - Y^{\gamma}_t }^2$ and it follows that
\begin{align*}
\frac{d}{dt} f_\gamma(t)
&= \frac{d}{dt} \norm{ 
x - y - M(\gamma) \int_0^t \left[ \nabla V(M(\gamma) X^{\gamma}_s) - \nabla V(M(\gamma) Y^{\gamma}_s) \right] ds 
}^2
\\
&= - 2 \ip{ \nabla V(M(\gamma) X^{\gamma}_t) - \nabla V(M(\gamma) Y^{\gamma}_t) , M(\gamma) (X^{\gamma}_t - Y^{\gamma}_t) } 
\\
&\le -2 \alpha f_\gamma(t).
\end{align*}
By Gronwall's inequality,
$
\sqrt{ \E\norm{ X^{\gamma}_1 - Y^{\gamma}_1 }^2 }
\le \exp(- \alpha) \norm{x - y}.
$
We have the upper bound
\begin{align}
\W_{\norm{\cdot}, 2}\left[ \P_\gamma(x, \cdot), \P_\gamma(y, \cdot) \right]
\le \sqrt{ \E\norm{ M(\gamma) X^{\gamma}_1 - M(\gamma) Y^{\gamma}_1 }^2 }
\le \exp(- \alpha) \norm{x - y}.
\label{eq:adaptive_diffusion_ub}
\end{align}
By \citep[Propositon 1]{durmus:2019}, $\int \norm{y}^2 \pi(dy)$ is finite and by Lemma~\ref{lemma:contract_containment}, weak containment \eqref{assumption:weak_containment} holds.
Lemma~\ref{lemma:contract_containment} then implies weak containment \eqref{assumption:weak_containment} and $(\norm{ X_n})_{n \ge 0}$ is uniformly integrable.
A similar argument as in Example~\ref{example:adaptive_ula} shows weak diminishing adaptation \eqref{assumption:weak_da}.
By Proposition~\ref{proposition:convergence_wasserstein}, the convergence in Wasserstein follows and Corollary~\ref{corollary:wlln} implies a weak law of large numbers for all bounded Lipschitz functions.
\end{proof}

\section{\hl{Connections to geometric drift and coupling conditions}}
\label{section:dm}

A general approach to satisfy the containment condition \eqref{assumption:weak_containment} is through a simultaneous version of the Weak Harris' Theorem \citep[Theorem 4.8]{hairer:etal:2011}.
Other similar convergence bounds for non-adapted Markov chains could be modified to simultaneous versions as well \citep{durmus:moulines:2015, qin:hobert:2022}.

\begin{theorem}
(Simultaneous Weak Harris' Theorem)
\label{theorem:simult_weak_harris}
Let $(\P_\gamma)_{\gamma \in \Y}$ be a family of Markov kernels on $\X$ with invariant probability measure $\pi$.
\begin{itemize}
\item (Simultaneous geometric drift) Suppose there is a Borel drift function $V : \Y \times \X \to [0, \infty)$ and constants $\lambda \in (0, 1)$, $K \in (0, \infty)$ such that for every $\gamma, x \in \Y \times \X$,
\[
\int_{\X} V(\gamma, x') \P_{\gamma}(x, dx') \le \lambda V(\gamma, x) + K.
\]
\item ($\rho$-contracting) Suppose there is a $\kappa \in (0, 1)$ such that for every $x, y \in \X$ with $\rho(x, y) < 1$ and every $\gamma \in \Y$,
\[
\W_{\rho \wedge 1}(\P_{\gamma}(x, \cdot), \P_{\gamma}(y, \cdot))
\le (1-\kappa) \rho(x, y).
\]
\item ($\rho-$small) Suppose for some constants $\alpha, \delta \in (0, 1)$ and every $\gamma \in \Y$,
\[
\sup_{x, y \in C_{\gamma}} \W_{\rho \wedge 1}(\P_{\gamma}(x, \cdot), \P_{\gamma}(y, \cdot))
\le 1 - \alpha
\]
where $C_{\gamma} = \{ x \in \X : V(\gamma, x) \le (1 + \delta) 2 K /(1 - \lambda) \}$.
\end{itemize}
Then there is an explicit $\alpha^* \in (0, 1)$ depending on $\alpha, \kappa, \lambda$ such that for every $t \in \Z_+$ and every $\gamma, x \in \Y \times \X$,
\[
\W_{\rho_\gamma}(\P^{t}_{\gamma}(x, \cdot), \pi)
\le (1 - \alpha^*)^{t} \sqrt{(1 + \beta^* V(\gamma, x) + (\alpha \wedge \kappa)/ [ 4(1 - \lambda)) ]}
\]
where $\rho_\gamma(u, v) = \sqrt{( \rho(u, v) \wedge 1) \left[ 1 + \beta^* V(\gamma, u) + \beta^* V(\gamma, v) \right] }$ and $\beta^* = ( \alpha \wedge \kappa )/(4K)$.
\end{theorem}

\begin{proof}
The argument is inspired by \citep[Theorem 4.8]{hairer:etal:2011}.
Fix $\gamma \in \Y$.
For $\beta > 0$, define $\rho_{\beta, \gamma}(x, y) = \sqrt{( \rho(x, y) \wedge 1) (1 + \beta ( V(\gamma, x) + V(\gamma, y) ) }$.
First assume for $x, y \in \X$, $\rho(x, y) \ge 1$ and $V(\gamma, x) + V(\gamma, y) > R$.
Now for $\delta > 0$, choose $R = (1 + \delta) 2 K /(1 - \lambda)$.
Then using the simultaneous drift condition
\begin{align*}
\W_{\rho_{\beta, \gamma}}( \P_{\gamma}(x, \cdot), \P_{\gamma}(x, \cdot) )^2
&\le 1 + \beta \P_\gamma V(\gamma, x) + \beta \P_\gamma V(\gamma, y)
\\
&\le 1 - \lambda + \lambda ( 1 + \beta V(\gamma, x) + \beta V(\gamma, y) ) + \beta 2 K
\\
&\le \left[ (1 - \lambda) \frac{1 + \beta 2 K/(1 - \lambda)}{1 + \beta R} + \lambda \right] \rho_{\beta, \gamma}(x, y)^2.
\end{align*}
Now assume for $x, y \in \X$, $\rho(x, y) \ge 1$ and $V(\gamma, x) + V(\gamma, y) \le R$.
Then using that $C_\gamma$ is $\rho-$small and the simultaneous drift condition and choosing $\beta \le \alpha/(4K)$,
\begin{align*}
\W_{\rho_{\beta, \gamma}}( \P_{\gamma}(x, \cdot), \P_{\gamma}(x, \cdot) )^2
&\le (1 - \alpha) \left[ 1 + \beta \P_\gamma V(\gamma, x) + \beta \P_\gamma V(\gamma, y) ) \right]
\\
&\le (1 - \alpha) \rho(x, y) \wedge 1 \left[ 1 + \lambda \beta \left( V(x) + V(y) \right) + \beta 2 K \right]
\\
&\le (1-\alpha/2) \rho_{\beta, \gamma}(x, y)^2.
\end{align*}
Next assume for $x, y \in \X$, $\rho(x, y) < 1$.
Then using $\rho$-contracting and the simultaneous drift condition with and $\beta \le \kappa/(4 K)$,
\begin{align*}
\W_{\rho_{\beta, \gamma}}( \P_{\gamma}(x, \cdot), \P_{\gamma}(x, \cdot) )^2
&\le (1 - \kappa) \rho(x, y)  \left[ 1 + \beta \P_\gamma V(\gamma, x) + \beta \P_\gamma V(\gamma, y)) \right]
\\
&\le \rho(x, y) \wedge 1  \left[ 1 - \kappa + \beta 2 K + (1 - \kappa) \lambda \beta ( V(\gamma, x) + V(\gamma, y) ) \right]
\\
&\le (1 - \kappa/2) \rho_{\beta, \gamma}(x, y)^2.
\end{align*}
\end{proof}

Theorem~\ref{theorem:simult_weak_harris} can be seen as an extension of simultaneous geometric drift and minorization conditions \citep[Theorem 18]{roberts:rosenthal:2007}.
We allow the drift function to depend on on the adapted tuning parameter and the metric $\rho$ is not restricted to the Hamming metric.
For drift functions $V$ constant in $\gamma \in \Y$ so that we can write $V : \X \to [0, \infty)$, we have the following result for the adaptive process if the conditions of Theroem~\ref{theorem:simult_weak_harris} hold true.

\begin{theorem}
\label{theorem:weak_harris_adaptive_conv}
Suppose the adaptive process $(\Gamma_t, X_t)_{t \ge 0}$ with initialization probability measure $\mu$ as in \eqref{eq:adaptive_process} constructed with Markov kernels $(\P_\gamma)_{\gamma \in \Y}$ satisfies weak diminishing adaptation~\eqref{assumption:weak_da}. 
Suppose the conditions of Theorem~\ref{theorem:simult_weak_harris} are satisfied for $(\P_\gamma)_{\gamma \in \Y}$ with a drift function $V$ constant in $\gamma \in \Y$.
Then
\[
\lim_{t\to\infty} \W_{\rho \wedge 1}\left( \mu\A^{(t)}, \pi \right)
= 0.
\]
\end{theorem}

\begin{proof}
Since $\rho \wedge 1$ is bounded, it will suffice to assume the adaptive process $(\Gamma_t, X_t)_t$ is initialized at $\gamma_0, x_0 \in \Y \times \X$.
The conclusion follows from Theorem~\ref{theorem:convergence} if we verify weak containment \eqref{assumption:weak_containment}.
The conditions of Theorem~\ref{theorem:simult_weak_harris} imply in order to satisfy weak containment, it suffices to show $( V(X_t) )_t$ is bounded in probability.
The geometric drift condition implies $( V(X_t) )_{t \ge 0}$ is bounded in probability by \citep[Lemma 15]{roberts:rosenthal:2007} since it is constant in $\gamma \in \Y$.
%The simultaneous geometric drift and a similar argument to Corollary~\ref{corollary:weighted_tv_convergence} extends weak convergence.
\end{proof}

To satisfy weak containment \eqref{assumption:weak_containment}, Theorem~\ref{theorem:weak_harris_adaptive_conv} can be weakened to subgeometric rates of convergence.
If there is a constant $M_0 > 0$, a Borel measurable function $V : \X \to [0, \infty)$ such that $( V(X_t) )_{t \ge 0}$ is bounded in probability and a rate function $r : \Z_+ \to [0, 1]$ with $\lim_{n \to \infty} r(n) = 0$ such that for all $n, t \in \Z_+$
\[
\W_{\rho \wedge 1}\left( \P_{\Gamma_t}^n( X_t, \cdot ), \pi \right) 
\le M_0 r(n) V(X_t)
\]
then it follows weak containment \eqref{assumption:weak_containment} holds.
For example, a polynomial drift condition is sufficient for $( V(X_t) )_{t \ge 0}$ to be bounded in probability \citep{bai:2011} and existing subgeometric rates of convergence for non-adapted Markov chains in Wasserstein distances can be modified to simultaneous versions \citep{butkovsky:2014,durmus:fort:moulines:2016}.

In certain cases, it can be difficult to find a drift function that does not change with $\gamma \in \Y$ and alternative techniques can be used to show (strong) containment.
One successful strategy here is to only apply adaptation on a compact or bounded set of the state space \citep[Theorem 21]{craiu:etal:2015}.
We will say adaptation is restricted to a Borel set $S \subset \X$ for the adaptive process $(\Gamma_t, X_t)_{t \ge 0}$ if for all $t \in \Z_+$, $X_{t-1} \not\in S$, $\Gamma_t = \Gamma_{t-1}$.
Our next result shows the benefit of Theorem~\ref{theorem:simult_weak_harris} with drift functions depending on the adapted tuning parameter if adaptation is restricted to a set.

\begin{proposition}
Suppose the adaptive process $(\Gamma_t, X_t)_{t \ge 0}$ with initialization probability measure $\mu$ as in \eqref{eq:adaptive_process} constructed with Markov kernels $(\P_\gamma)_{\gamma \in \Y}$ satisfies weak diminishing adaptation~\eqref{assumption:weak_da}. 
Suppose the conditions of Theorem~\ref{theorem:simult_weak_harris} are satisfied with Markov kernels $(\P_\gamma)_{\gamma \in \Y}$ and drift function $V(\cdot, \cdot)$.
Additionally, assume $(\Gamma_t, X_t)_{t \ge 0}$ has adaptation restricted to the Borel set $S \subset \X$ and 
$\sup_{x \in S} \sup_{t \in \Z_+} \E\left[ V(\Gamma_{t + 1}, X_{t + 1}) \mid \H_{t}, X_{t} = x \right] < \infty$.
Then
\[
\lim_{t\to\infty} \W_{\rho \wedge 1}\left( \mu\A^{(t)}, \pi \right)
= 0.
\]
\end{proposition}

\begin{proof}
It will suffice to assume the adaptive process $(\Gamma_t, X_t)_t$ is initialized at $\gamma_0, x_0 \in \Y \times \X$, and by Theorem~\ref{theorem:convergence}, we may show that $(V(\Gamma_t, X_t))_t$ is bounded in probability.
By the assumptions, we can assume $\E\left[ V(\Gamma_t, X_t) I_{S}(X_{t-1}) \right] \le K$.
We have the upper bound
\begin{align*}
\E\left[ V(\Gamma_t, X_t) \right]
&\le \E\left[ V(\Gamma_t, X_t) I_{S}(X_{t-1}) \right] 
+ \E\left[ V(\Gamma_t, X_t) I_{S^c}(X_{t-1}) \right]
\\
&\le K + \E\left[ V(\Gamma_{t-1}, X_t) I_{S^c}(X_{t-1}) \right]
\\
&\le 2K + \lambda \E\left[ V(\Gamma_{t-1}, X_{t-1}) \right].
\\
&\le 2K/(1 - \lambda) + V(\gamma_0, x_0).
\end{align*}
Therefore, $( V(\Gamma_t, X_t) )_t$ is bounded in probability by Markov's inequality.
\end{proof}

\section{Concluding remarks}
\label{section:conclusion}

This article developed weak convergence of adaptive MCMC processes under general conditions which is suited to situations where convergence in total variation is inadequate. 
One motivation is adapting the tuning parameters of reducible Markov chains where the traditional theory of adaptive MCMC may not be applied.
Another application of the developed theory can be used to analyze adaptive MCMC processes where (strong) containment is difficult to show but weak containment may be more tractable.
The weak law of large numbers developed here can be seen as an extension of \citep[Theorem 23]{roberts:rosenthal:2007} and appears of practical interest for the reliability and stability of adaptive MCMC simulations widely used in statistics and machine learning.

There are many future research directions worthy of pursuit.
While the developed theory for weak convergence allows a Markov process in continuous time to be adapted at discrete times, the proof techniques do not appear limited to discrete-time adaptive processes\hl{. In particular,} a precise formulation of an adaptive MCMC process in continuous-time with similar convergence results appears feasible.
It also appears some techniques for the convergence results developed here might be extended to develop quantitative convergence rates in Wasserstein distances or mixing times for adaptive MCMC, but would require stronger assumptions on the adaptation plan.
Another interesting direction is to try to generalize Theorem~\ref{theorem:convergence} to hold for general possibly unbounded Wassserstein distances by adapting weak containment and weak diminishing adaptation to hold using general Wasserstein distances.
%It also appears alternative assumptions for the weak law of large numbers may exist and furthermore, necessary conditions for a law of large numbers to hold are of further interest for adaptive MCMC.

\section*{Acknowledgements}
\hl{The authors would like to thank the Associate Editor and the two anonymous referees for their insightful comments in helping to improve this article.}

\section*{Funding information}
\hl{This work was partially funded by NSERC Discovery Grant RGPIN-2019-04142, and by a postdoctoral fellowship at the University of Toronto.}

%\section*{Competing interests}
%\hl{There were no competing interests to declare which arose during the preparation or publication process of this article.}

\bibliography{references.bib}

\end{document}